\documentclass[11pt]{amsart}  	
\usepackage[margin=2.5cm]{geometry}

\usepackage[all]{xy}						
\usepackage{bbm}
\usepackage{amssymb}
\usepackage{amscd,latexsym,amsthm,amsfonts,amssymb,amsmath,amsxtra}
\usepackage[mathscr]{eucal}
\usepackage[colorlinks]{hyperref}
\usepackage{hyperref}
\usepackage{cleveref}
\usepackage{mathtools}
\usepackage{cite}
\usepackage{chngcntr}
\usepackage{comment}
\usepackage{pictexwd,dcpic}
\numberwithin{equation}{subsection}
\newcounter{keepeqno}

\newcommand{\trivial}[2][]{\if\relax\detokenize{#1}\relax {\color{red} \vspace{0em} {[} #2 {]}} \else\ifx#1h\ifcsname showtrivial\endcsname{\color{orange} \vspace{0em} {[}  #2 {]}}\fi\else {\red Wrong argument!} \fi\fi}

\hypersetup{linkcolor=black, citecolor=black}

\newcommand{\BC}{{\mathbb {C}}}

\newcommand{\BG}{{\mathbb {G}}}

\newcommand{\BR}{{\mathbb {R}}}

\newcommand{\BZ}{{\mathbb {Z}}}

\newcommand{\CF}{{\mathcal {F}}}
\newcommand{\CG}{{\mathcal {G}}}
\newcommand{\CH}{{\mathcal {H}}}

\newcommand{\CS}{{\mathcal {S}}}

\newcommand{\WW}{{\mathrm {W}}}

\newcommand{\Fw}{{\mathfrak {w}}}

\newcommand{\RM}{{\mathrm {M}}}

\newcommand{\RO}{{\mathrm {O}}}

\newcommand{\Cent}{{\mathrm{Cent}}}

\newcommand{\disc}{{\mathrm{disc}}}

\newcommand{\Gal}{{\mathrm{Gal}}}
\newcommand{\GL}{{\mathrm{GL}}}

\newcommand{\Hom}{{\mathrm{Hom}}}

\renewcommand{\Im}{{\mathrm{Im}}}
\newcommand{\ind}{{\mathrm{ind}}}
\newcommand{\Ind}{{\mathrm{Ind}}}

\newcommand{\Mp}{{\widetilde{\mathrm{Sp}}}}

\newcommand{\Or}{{\mathrm{O}}}

\newcommand{\Smod}{\mathcal{S}\mathrm{mod}}
\newcommand{\nInd}{\mathfrak{Ind}} 
\newcommand{\unind}{\mathfrak{ind}} 
\newcommand{\unInd}{\mathfrak{Ind}} 

\newcommand{\diag}{{\mathrm{diag}}}

\newcommand{\SO}{{\mathrm{SO}}}

\newcommand{\Sym}{{\mathrm{Sym}}}
\newcommand{\sgn}{{\mathrm{sgn}}}
\newcommand{\Sp}{{\mathrm{Sp}}}
\newcommand{\Stab}{{\mathrm{Stab}}}

\newcommand{\Span}{{\mathrm{Span}}}
\newcommand{\gen}{\mathrm{gen}}

\newcommand{\wh}{\widehat}
\newcommand{\wt}{\widetilde}

\def\diag{{\rm diag}}

\def\std{\rm std}

\def\Irr{\mathrm{Irr}}

\def\t{\tilde}

\def\temp{\mathrm{temp}}

\def\Hei{\mathrm{Hei}}
\def\J{\mathrm{J}}
\def\calJ{\mathcal{J}}

\def\MVW{\mathrm{MVW}}
\def\even{\mathrm{even}}
\def\odd{\mathrm{odd}}

\newtheorem{thm}{Theorem}[subsection]
\newtheorem{defin}[thm]{Definition}
\newtheorem{rmk}[thm]{Remark}

\newtheorem{pro}[thm]{Proposition}
\newtheorem{lem}[thm]{Lemma}
\newtheorem{cor}[thm]{Corollary}

\makeatletter

\newcommand{\Rmnum}[1]{\expandafter\@slowromancap\romannumeral #1@}
\makeatother

\long\def\ccheng#1{{{\color{red}CC: #1}}}

\long\def\zjl#1{{{\color{blue}ZJL: #1}}}

\begin{document}

\title{Fourier--Jacobi models for real symplectic-metaplectic groups: the basic case}

\author{Cheng Chen}
\address{Université Paris Cité, Sorbonne Université, CNRS, IMJ-PRG, Bâtiment Sophie Germain,
8 place Aurélie Nemours, F-75013 Paris, France}
\email{cheng.chen@imj-prg.fr}

\author{Rui Chen}
\address{School of Mathematical Sciences, Zhejiang University, Zijingang Campus, 866 Yuhangtang Road, Hangzhou 310058, China}
\email{rchenmat@zju.edu.cn}

\author{Jialiang Zou}
\address{
Department of Mathematics, Massachusetts Institute of Technology, 77 Massachusetts Avenue, Cambridge, MA 02139, USA
}
\email{jlzou@mit.edu}


\begin{abstract}
In this paper, we generalize the method in \cite{gan2016gross}\cite{atobe2018local} to the field of real numbers and prove the basic tempered case of the local Gan-Gross-Prasad conjecture for Fourier-Jacobi models of symplectic-metaplectic groups, based on the tempered case of the conjecture for Bessel models proved in \cite{chen2022local}.
\end{abstract}

\maketitle

\tableofcontents

\section{Introduction}
The aim of this paper is to prove the basic tempered case of the local Gan-Gross-Prasad conjecture \cite{gan2012symplectic} for symplectic–metaplectic Fourier–Jacobi models over the field $\BR$.

\vskip 5pt

Let $W$ be a symplectic vector space over $\BR$. Let $\Sp(W)$ be the associated symplectic group and $\Mp(W)$ the metaplectic group, i.e., the unique non-split double cover of $\Sp(W)$. For a non-trivial additive character $\psi$ of $F$, let $\omega_{\psi}$ denote the Weil representation of $\Mp(W)$ associated with $\psi$ (see \Cref{Section Weil representation}). We set 
\[
G = \Sp(W) \times \Mp(W) \quad \mbox{and} \quad H = \Delta\Mp(W),
\]
where $\Delta\Mp(W)$ denotes the diagonal embedding of $\Mp(W)$ into $G$.
We denote by $\Irr(\Sp(W))$ the set of 
irreducible admissible Casselman–Wallach representations of $\Sp(W)$ and  $ \Irr_{\gen}(\Mp(W))$ the set of genuine irreducible admissible Casselman–Wallach representations of $\Mp(W)$. For $\pi \in \Irr(\Sp(W))$ and $\wt{\pi} \in \Irr_{\gen}(\Mp(W))$, we define the multiplicity
\begin{equation}\label{fj MODEL}
    m\!\left(\pi \boxtimes \wt{\pi}\right) 
= \dim \Hom_{H}\!\left(\pi \boxtimes \wt{\pi}, \, \omega_{\psi}\right).
\end{equation}
The well-known multiplicity-one theorem asserts that
\[
m\!\left(\pi \boxtimes \wt{\pi} \right) \leqslant 1.
\]
This theorem was proved in \cite{sun2012multiplicity,liu2013uniqueness}. The local Gan–Gross–Prasad conjecture provides a refinement of this multiplicity-one theorem, by precisely describing the behavior of 
$m\!\left(\pi \boxtimes \wt{\pi} \right)$ for $\pi, \wt{\pi}$ belonging to generic L-packets. 

\vskip 5pt

To state the precise conjecture, we first recall the local Langlands correspondence (``LLC" for short). We denote by $\WW_\BR$ the Weil group of $\BR$. For a connected reductive group $G$ over $\BR$ and a local L-parameter
\[
\phi:\WW_\BR \longrightarrow {{}^LG},
\]
there is a finite subset $\Pi_{\phi}(G)$ of $\Irr(G(\BR))$, called the L-packet of $G$ associated to $\phi$. Following the work of Vogan \cite{vogan1993local}, the Vogan L-packet attached to $\phi$ is defined as
\[
\Pi_{\phi} = \bigsqcup_{G'}\, \Pi_{\phi}(G'),
\]
where $G'$ runs over all pure inner forms of $G$.  
The advantage of working with $\Pi_\phi$ is that, after fixing a Whittaker datum $\Fw$ for a quasi-split pure inner form $G^*$ of $G$, there exists a bijection
\[
\calJ_{\Fw}: \Pi_{\phi} \;\longleftrightarrow\; \widehat{\CS_{\phi}}, \qquad \pi \;\longleftrightarrow\; \eta_{\pi},
\]
where $\widehat{\CS_{\phi}}$ denotes the set of characters of the component group 
\[
\CS_{\phi} \coloneqq \pi_0\left(\Cent_{\wh{G}}(\Im(\phi))\right).
\]

\vskip 5pt

In \cite{adams1998genuine}, Adams and Barbasch also extended the LLC to the metaplectic group $\Mp(W)$, by using the local Shimura–Waldspurger correspondence. We briefly recall their construction. Assume that $\dim W = 2n$, and let $V^+$ be a split quadratic space over $\BR$ with $\dim V^+ = 2n+1$ and $\disc(V^+)=1$. The Shimura–Waldspurger correspondence gives a bijection
\[
\theta_{V,W,\psi}: \Irr_{\gen}(\Mp(W)) \;\longleftrightarrow\; 
\coprod_{\substack{\dim V = 2n+1 \\ \disc(V) = 1}} \Irr(\SO(V)).
\]
Note that these quadratic spaces $V$ with $\dim V = 2n+1$ and $\disc(V)=1$ exactly parametrize all pure inner forms of $\SO(V^+)$. Therefore, one can define the L-group of $\Mp(W)$ to be the same as that of $\SO(V^+)$. Given an L-parameter $\phi$ of $\Mp(W)$, one can also regard $\phi$ as an L-parameter of $\SO(V^+)$. We then define the (Vogan) L-packet to be  
\[
\Pi_{{\phi}}(\Mp(W)) = \left\{\pi\in\Irr_{\gen}(\Mp(W))\,\big|\,\theta_{V,W,\psi}(\pi)\in\Pi_{\phi}\right\}.
\]
For $\pi \in \Pi_{{\phi}}(\Mp(W))$, we set
\[
\eta_{{\pi}} := \calJ_{\Fw}\left(\theta_{V,W,\psi}(\pi)\right) \in  \widehat{\CS_{\phi}}.
\]
This provides a natural bijection
\[
\pi \in \Pi_{{\phi}}(\Mp(W)) \quad \longleftrightarrow \quad \eta_\pi \in \widehat{\CS_{{\phi}}}.
\]
The main result of this paper is the so-called ``\emph{basic tempered case}'' of the local Gan–Gross–Prasad conjecture \cite[Conjecture 17.1, 17.3]{gan2012symplectic} for symplectic–metaplectic Fourier–Jacobi models.
\begin{thm}\label{conj: GP in introduction}
The following statements hold:
\begin{enumerate}
    \item \textbf{Multiplicity-one:} For every tempered L-parameter $\phi$ of $\Sp(W)$ and every tempered L-parameter $\wt{\phi}$ of $\Mp(W)$, we have
    \[
        \sum_{\pi \in \Pi_{\phi}} \; \sum_{\wt{\pi} \in \Pi_{\wt{\phi}}} 
        m\!\left(\pi \boxtimes \wt{\pi} \right) = 1.
    \]
    \item \textbf{Epsilon-dichotomy:} The unique pair $(\pi,\wt{\pi}) \in \Pi_{\phi} \times \Pi_{\wt{\phi}}$ such that
    \[
        m\!\left(\pi \boxtimes \wt{\pi}\right) = 1
    \]
    corresponds to the distinguished characters of component groups
    \[
        \eta_{\pi} \times \eta_{\wt{\pi}} = \eta_{\wt\phi} \times \eta_{\phi},
    \]
    where $\eta_{\wt\phi} \times \eta_{\phi}$ is defined in terms of local root numbers as in (\ref{distigulished character}). 
\end{enumerate}
\end{thm}

\vskip 5pt

We should remark that the local Gan–Gross–Prasad conjecture is originally formulated in a much more general context. There are several different settings, which can be broadly divided into \emph{Bessel models} and \emph{Fourier-Jacobi models}.  
\begin{itemize}
    \item The \textbf{Bessel models} concern special orthogonal groups and unitary groups. For special orthogonal groups, these are often referred to as the \emph{local Gross–Prasad conjecture}, originally introduced in \cite{gross1992decomposition,gross1994irreducible}.  

    \vskip 5pt
    
    \item The \textbf{Fourier–Jacobi models} concern symplectic–metaplectic groups and unitary groups. In the symplectic-metaplectic case, one can consider groups of the form $\Sp(W) \times \Mp(W')$, where $W$ and $W'$ are not necessarily of the same dimension; also, one can consider generic L-parameters instead of tempered L-parameters. The case we are considering in this paper, namely, when $W=W'$ and L-parameters are tempered, is called the basic tempered case (of Fourier-Jacobi models).
\end{itemize}
\vskip 5pt
The general form of this conjecture can be reduced to the basic tempered cases: when the base field is non-Archimedean, this was done in \cite[Theorem 15.1, 16.1]{gan2012symplectic}; when the base field is Archimedean, this was done in the work of the first author \cite{chen2023multiplicity}. 

\vskip 5pt

Over the non-Archimedean local field, the local Gan-Gross-Prasad conjecture is known in all cases, by the works of Waldspurger \cite{waldspurger2010formule,waldspurger2012calcul,waldspurger2012conjecture,waldspurger2012formule,waldspurger2012variante},  M{\oe}glin–Waldspurger \cite{moeglin2012conjecture}, Beuzart–Plessis\cite{beuzart2014expression,beuzart2016conjecture}, Gan–Ichino \cite{gan2016gross} and Atobe \cite{atobe2018local}.
\vskip 5pt
Over the complex number field $\BC$, since the (Vogan) L-packet contains only one element, the multiplicity formula in \cite[Theorem A]{chen2023multiplicity} implies the local Gan-Gross-Prasad conjecture.

\vskip 5pt

Over the real number field $\BR$, the current status of the conjecture is as follows:

\begin{itemize}
\item For Bessel models of unitary groups, Beuzart–Plessis proved the multiplicity one part of the conjecture for tempered L-parameters in \cite{beuzart2014expression,beuzart2016conjecture}. H. He proved the full conjecture for discrete L-parameters in \cite{MR3681395} and Xue proved the full conjecture for tempered L-parameters in \cite{MR4568696}. Both \cite{MR3681395} and \cite{MR4568696} use theta correspondence as the main tool. These results are extended to generic L-parameters by Xue in \cite{xue2020bessel} using the Schwartz homology.

\vskip 5pt

\item The local Gross–Prasad conjecture for tempered L-parameters was proved in \cite{luo2020local,chen2022local}, following Waldspurger’s approach. These results are extended to generic L-parameters by the first author in \cite{chen2023multiplicity}.

\vskip 5pt

\item For Fourier–Jacobi models of unitary groups, Xue proved the conjecture in \cite{MR4796148}. He first used (almost) equal rank theta correspondence to reduce the basic tempered case to the Bessel models for unitary group, and then extended to generic L-parameters using a similar strategy as in \cite{xue2020bessel}. 
\end{itemize}
\vskip 5pt
The only remaining case is the Fourier–Jacobi model for symplectic–metaplectic groups. In this paper, we establish the conjecture for the basic tempered case. Together with the reduction results of \cite{chen2023multiplicity}, this settles the last outstanding case of the local Gan–Gross–Prasad conjecture, thereby bringing this long program to completion. 

\vskip 5pt

Next, we outline our strategy for proving \Cref{conj: GP in introduction}. The method of \cite{gan2016gross,atobe2018local} (over non-archimedean local fields) allows one to reduce the Fourier–Jacobi models to the Bessel models via the (almost) equal rank theta correspondence. Two key ingredients are needed in this approach:

\begin{itemize}
    \item the irreducibility of the ``big theta lifts" of tempered representations; and 

    \vskip 5pt
    
    \item the description of theta lifts in terms of L-parameters (the so-called Prasad's conjecture).  
\end{itemize}
\vskip 5pt
Over $\BR$, however, the irreducibility of the ``big theta lifts" in the (almost) equal rank case is not known in general. To circumvent this difficulty, we instead employ the stable range theta lift, for which the irreducibility of the ``big theta lifts" has been established in \cite{MR3305312}. Using the stable range theta lift, we are led naturally to certain non-tempered Gan–Gross–Prasad problems (see \cite{MR4190046}), both in the Bessel and Fourier–Jacobi cases. We then relate these non-tempered cases to the tempered ones by the mean of non-tempered reductions, which generalizes the arguments of \cite[\S 1.4]{moeglin2012conjecture} from the non-Archimedean setting to the real number field $\BR$. We hope that our approach here can shed some light for non-tempered Gan-Gross-Prasad problems. We should mention that Zhe Li and Shanwen Wang are also considering the basic tempered case of Fourier-Jacobi models, but with a different approach.

\vskip 5pt

This paper is organized as follows. In Section~\ref{section LLC}, we review the local Langlands correspondence for the groups that appear in this work.  
In Section~\ref{section: GGP conjecture} we recall the statements of the local Gross–Prasad conjecture and the local Gan–Gross–Prasad conjecture for the symplectic–metaplectic Fourier–Jacobi models.  
In Section~\ref{sec theta} we introduce some basic notions of Schwartz analysis and the properties of the theta correspondence that we will use.  
In Section~\ref{section: stable 1}, we prove the uniqueness part of the conjecture, namely: if there exists a pair of representations $(\pi,\wt{\pi})$ in a Vogan L-packet that admit a nonzero Fourier–Jacobi model, then they must be the pair of representations determined by the distinguished characters defined in (\ref{distigulished character}).  
Finally, in Section~\ref{section: stable 2} we establish the existence of the Fourier–Jacobi model, i.e., there does exist a pair $(\pi,\wt{\pi})$ in each Vogan L-packet with a nonzero Fourier–Jacobi model.  
Combining the results of Sections~\ref{section: stable 1} and~\ref{section: stable 2}, we complete the proof of \Cref{conj: GP in introduction}.

\vskip 5pt

\subsection*{Notation}

We end this introduction by setting some notations and conventions. Throughout the paper, we fix a non-trivial additive character $\psi$ of $\mathbb R$. For any $a\in \BR^\times$, we define $\psi_a(x)=\psi(ax)$ for $x\in \BR$. For a Nash group $G$, its closed Nash subgroup $H$, and a representation $\sigma$ of $H$, we denote by $\Ind_H^G(\sigma)$ and $\ind_H^G(\sigma)$ the usual normalized smooth induction and the normalized Schwartz induction from $H$ to $G$, respectively. In our proof of the main theorem, we will also make use of unnormalized inductions, which we shall denote by $\nInd$ and $\unind$ respectively. Since we will work with classical groups, we shall employ the standard notations
\[
\pi_1\times\pi_2\quad \mbox{and}\quad\tau\rtimes\pi
\]
to denote normalized parabolic inductions of general linear groups and classical groups. All tensor products in this paper should be understood as completed tensor products.

\vskip 5pt

\subsection*{Acknowledgement}

The first author is supported by the European Union’s Horizon 2020 research and innovation programme under the Marie Skłodowska-Curie grant agreement No. 101034255. The third author is supported by the Simons Foundation. The authors thank Zhibin Geng, Zhe Li, Shanwen Wang, Hang Xue and Hao Ying for helpful discussions. 

\vskip 10pt

\section{Local Langlands correspondence}\label{section LLC}
In this section, we recall the notion of Vogan L-packets for special orthogonal groups, symplectic groups, and metaplectic groups over $\BR$ following \cite{gan2012symplectic} and \cite{vogan1993local}.

\vskip 5pt

\subsection{Weil groups and L-parameters}
Recall that the Weil group $\WW_{\mathbb R}$ of $\mathbb R$ is defined to be
\[
\WW_{\BR}:=
\BC^\times \sqcup \BC^\times j,
\]
with $j^2=-1$ and $j\cdot z\cdot j^{-1}=\overline{z}$ for $z\in \BC^{\times}$. There is an isomorphism $\WW_{\mathbb R}^{\mathrm {ab}}\cong \mathbb R^\times$, where the superscript ``${\rm ab}$" means the abelianization. Denote by $\mathbbm 1$ and $\sgn$ the trivial and sign characters of $\mathbb R^{\times}$ and view them as characters of $\WW_{\mathbb R}$ via the isomorphism above. For $x\in \frac{1}{2}\mathbb Z$, we define a two-dimensional induced representation
\[
D_{x}=\Ind_{\mathbb C^\times }^{\WW_{\mathbb R}} \left(\frac{z^{2x}}{|z|^x}\right).
\]
Then $D_{x}\cong D_{-x}$ and they are irreducible unless $x=0$, in which case $D_{0}\cong \mathbbm 1 \oplus \sgn$. Moreover, $D_{x}$ is orthogonal if and only if $x \in \mathbb Z$, and $D_{x}$ is symplectic if and only if $x \in \frac{1}{2}+\mathbb Z$.
Any irreducible self-dual representation of $\WW_{\BR}$ is either

\begin{itemize}
	\item 1-dimensional, which factors through $\WW_{\mathbb R} ^{\rm ab}\cong \mathbb R^*$ and is therefore equal to the trivial character $\mathbbm 1$ or the sign character $\sgn$; or

    \vskip 5pt
    
	\item 2-dimensional, which is isomorphic to $D_{x}$ for some $x\in \frac{1}{2}\mathbb Z\,\backslash \{0\}$. 
\end{itemize} 
\vskip 5pt
For a finite-dimensional representation $M$ of $\WW_{\BR}$, we set $\epsilon(M,\psi)
$ to be the local root number defined in \cite[\S 16]{jacquet2009archimedean}. These local root numbers are determined by the following rules:

\begin{enumerate}
    \item For two representations $M_1$ and $M_2$ of $\WW_\BR$, one has
    \begin{equation}\label{equ: epsilon add}
        \epsilon(M_1\oplus M_2,\psi)=\epsilon(M_1,\psi)\cdot\epsilon(M_2,\psi);
    \end{equation}

    \vskip 5pt
    
    \item Let $\psi_a(x)=\psi(ax)$, then we have
    \begin{equation}
        \epsilon(M,\psi_a)=\det M(a)\cdot \epsilon(M,\psi);
    \end{equation}

    \vskip 5pt
    
    \item Set $\psi_0(x)=e^{2\pi ix}$, then the local root numbers associated with irreducible representations of $\WW_\BR$ are given by
    \begin{equation}\label{equ: epsilon irr}
        \epsilon(M,\psi_0)=\begin{cases}
        (-i)^\varepsilon&\text{if }M=\sgn^{\varepsilon}|\cdot|^s,\\[5pt]
        (-i)^{2|x|+1}&\text{if } M=D_x|\cdot|^{s},
    \end{cases}\quad
    \end{equation}
    where $\varepsilon\in\{0,1\}$, $x\in \frac{1}{2}\BZ$, and $s\in\BC$.
\end{enumerate}
\vskip 5pt

Let $G$ be a connected reductive algebraic group over $\mathbb R$. We denote by $\wh{G}$ the Langlands dual group of $G$, and by ${}^LG:= \wh{G}\rtimes \Gal(\BC/\BR)$ the L-group of $G$. We say that a continuous homomorphism
\[
\phi:\WW_{\BR}\longrightarrow {}^LG
\]
is an L-parameter of $G$ if:

\begin{itemize}
    \item elements in the image $\Im(\phi)\subset {}^LG$ are semi-simple;

    \vskip 5pt
    
    \item the composition of $\phi$ with the projection map ${}^L G\rightarrow \Gal(\BC/\BR)$ is the natural projection map $\WW_{\BR}\rightarrow \Gal(\BC/\BR)$.
\end{itemize}
\vskip 5pt
Let $\Phi(G)$ be the set of the $\wh{G}$-conjugacy classes of L-parameters for $G$.
 An L-parameter $\phi$ is called \textbf{tempered} if $\Im(\phi)$ is bounded.    
We denote by $\Phi_{\temp}(G)$ the subset of the conjugacy class of tempered L-parameters of $G$. 
For an L-parameter $\phi$, we denote by $S_\phi=\Cent_{\wh{G}}(\Im(\phi))$ the centralizer of the image of $\phi$ in $\wh{G}$, and set $\mathcal{S}_\phi := \pi_0\left(S_\phi\right)$.

\vskip 5pt

\subsection{Local Langlands correspondence for classical groups}

Let $G$ be a connected reductive group over $\BR$, the celebrated local Langlands correspondence, established by R. Langlands in \cite{langlands1973the} 
and Shelstad in \cite{MR2454336, MR2448289, MR2581952}, asserts that:

\begin{enumerate}
    \item There is a finite to one surjective map
\[
    \mathcal {L}: \bigsqcup_{G'} \,\Irr(G') \longrightarrow \Phi(G),
\]
where $G'$ runs over the isomorphism classes of pure inner forms of $G$, and $\Irr(G')$ is the set of irreducible Casselman--Wallach representations of $G'(\BR)$ (see \cite{casselman1989canonical}, \cite{wallach1994real}).
For each $\phi\in \Phi(G)$, we denote by $  \Pi_{\phi}=\mathcal {L}^{-1}(\phi)$ the Vogan L-packet associated to $\phi$ and $\Pi_{\phi}(G')=  \Pi_{\phi}
\cap \Irr(G')$ the local L-packet of $G'$ associated to $\phi$. 
The map $\mathcal {L}$ preserves temperedness, i.e., $\pi\in \Irr(G')$ is tempered if and only if $\mathcal {L}(\pi)$ is a tempered L-parameter. 

\vskip 5pt

\item Let $G^*$ be a quasi-split pure inner form of $G$. A \textbf{Whittaker datum} of $G^*$ is a conjugacy class of pairs $\Fw = (B,\mu)$, where $B$ is an $\BR$-rational Borel subgroup of $G^*$ with unipotent radical $U$, and $\mu$ is a generic character of $U(\BR)$. Fix a Whittaker datum $\Fw$ of $G^*$, then there exists a bijection 
\[
    \calJ_\Fw:   \Pi_{\phi} \longrightarrow \Irr\left(\mathcal{S}_{\phi}\right),
\]
such that the unique $\Fw$-generic representation inside $\Pi_{\phi}$ corresponds to the trivial character of $\mathcal{S}_{\phi}$. 
\end{enumerate}
\vskip 5pt

In the case that $G$ is a classical group, given an L-parameter $\phi\in \Phi(G)$, by composing $\phi$ with the standard representation of ${}^LG$, we can often regard $\phi$ as a semi-simple representation of $\WW_{\BR}$ with some additional self-duality condition. Now we give a more concrete description of L-parameters, component groups, pure inner forms and Whittaker data of special orthogonal and symplectic groups following \cite[\S 10, 11]{gan2012symplectic}.

\vskip 5pt

\subsubsection{\bf The symplectic group}\label{sympletic group}
Let $(W, \langle \,,\,\rangle)$ be a $2n$-dimensional symplectic space and $G=\Sp(W)$. Then $\wh{G}=\SO_{2n+1}(\BC)$ and ${}^LG=\SO_{2n+1}(\BC)\times \Gal(\BC/\BR)$.

\begin{itemize}
\item An L-parameter $\phi$ of $G$ gives rise to an orthogonal representation $M$ of $\WW_{\BR}$ with $\dim M=2n+1$ and $\det(M)=1$. Let $\underline \Phi(G)$ be the set of isomorphism classes of $(2n+1)$-dimensional orthogonal representations of $\WW_\BR$ with determinant $1$. 
Then the map $\phi\mapsto  M$ gives a bijection between $\Phi(G)$ and $\underline \Phi(G)$. 
In later sections, we often refer to $M$ as an L-parameter of $G$ via the above bijection and put $\Pi_{M}(G') = \Pi_{\phi} (G')$.

\vskip 5pt

\item  We denote by $C_M$ the centralizer of the representation $M$ and $A_M=\pi_0(C_M)$ be the component group of $C_M$. We also let 
\[
C_M^+=\{a\in C_M|\det(a)=1\}
\]
and $A_M^+$ be the image of $C_M^+$ in $A_M$. Since $M$ is orthogonal, it admits a decomposition 
\[
   M = P + \sum_i m_iM_i + P^\vee,
\]
where $M_i$ are pairwise distinct irreducible orthogonal subrepresentations of $M$, $m_i$ is the multiplicity of $M_i$ in $M$, and any subrepresentation of $P$ is not orthogonal. Then 
\[
    A_M \cong  \prod_i \left(\BZ/2\BZ\right) a_i
\]
with a canonical basis $\{a_i\}$ corresponding to orthogonal subrepresentations $\{M_i\}$, and  $A_M^+$ is the kernel of the map
\[
    A_M \longrightarrow \BZ/2\BZ,\quad a_i \mapsto (-1)^{\dim M_i}.
\]
In particular,  $A_M^+$ is an index $2$ subgroup of $A_M$ since there is some $M_i$ that is 1-dimensional. According to \cite[Theorem 8.1]{gan2012symplectic}, we have $\mathcal{S}_\phi \cong A_M^+$. For each $c\in \mathbb R^\times$, we define a character $\delta_{M,c}$ of $C_M^+$ by setting
\begin{equation}\label{delta}
  \delta_{M,c}(a) = \det\left(M^a\right)(c),
\end{equation}
where $a\in C_M^+$, and $M^a$ is the $(-1)$-eigenspace of $a$ in $M$. This character $\delta_{M,c}$ factors through $A_M^+$, and we regard it as a character of $A_M^+$.

\vskip 5pt

\item The only pure inner form of $G$ is itself, so we have
\[
    \Pi_{\phi} = \Pi_\phi(G).
\]

\vskip 5pt

\item By \cite[\S 12]{gan2012symplectic}, the set of Whittaker data for $G = \Sp(W)$ is in bijection with the set of $\mathbb R^{\times 2}$-orbits of non-trivial characters $\psi$ of $\mathbb R$. In this paper, we fix a non-trivial character $\psi$ and denote by $\mathfrak{w}$ the Whittaker datum of $G = \Sp(W)$ corresponding to $\psi$.

\vskip 5pt

\item Let $\pi\in \Pi_M$. It follows from \cite[Theorem 4.9]{MR3194648}) that $\pi^\vee\in \Pi_{M}$, and
\begin{equation}\label{dual pi}
     \calJ_{\Fw}(\pi^\vee) =\calJ_{\Fw}(\pi)\otimes \delta_{M,-1}. 
\end{equation}
\end{itemize}

\vskip 5pt

\subsubsection{\bf The odd special orthogonal group}\label{odd special orthogonal group}

Let $(V,q)$ be a $(2n+1)$-dimensional quadratic space and $G={\SO}(V)$.  Recall that the discriminant of $V$ is defined by the formula
\begin{equation}\label{disc V odd}
 \disc(V)=(-1)^n \det \left(\langle e_i, e_j\rangle\right)_{i,j}\in \BR^\times/\BR^{\times2},
\end{equation}
where $\{e_i\}$ is an orthogonal basis of $V$, and 
\begin{equation}\label{symmteric bilinear form}
  \langle x,y\rangle= \frac{1}{2}\left(q(x+y)-q(x)-q(y)\right),\quad \mbox{for } x,y\in V
\end{equation}
is the symmetric bilinear form associated to $q$. In this case, we have $\wh{G}=\Sp_{2n}(\BC)$ and ${}^LG=\Sp_{2n}(\BC)\times \Gal(\BC/\BR)$. 

\begin{itemize}
\item An L-parameter $\phi$ of $G$ gives rise to a symplectic representation $M$ of $\WW_{\BR}$ with $\dim M=2n$. Let $\underline \Phi(G)$ be the set of isomorphism classes of $2n$-dimensional symplectic representations of $\WW_\BR$. Then the map $\phi\mapsto M$ gives a bijection between $\Phi(G)$ and $\underline \Phi(G)$. In later sections we often refer to $M$ as an L-parameter of $G$ via the above bijection and put $\Pi_{M}(G) = \Pi_{\phi} (G)$.

\vskip 5pt

\item We define the groups $C_M, C_M^+, A_M, A_M^+$ in the same way as in the case of symplectic groups \Cref{sympletic group}. Since $M$ is symplectic, it admits a decomposition 
\[
   M = P + \sum_i m_iM_i + P^\vee,
\]
where $M_i$ are pairwise distinct irreducible symplectic subrepresentations of $M$, $m_i$ is the multiplicity of $M_i$ in $M$, and any subrepresentation of $P$ is not symplectic. Then 
\[
    A_M \cong  \prod_i \left(\BZ/2\BZ\right) a_i,
\]
with a canonical basis $\{a_i\}$ corresponding to the symplectic subrepresentations $\{M_i\}$, and  $A_M^+=A_M$ in this case since every $M_i$ is 2-dimensional. According to \cite[Theorem 8.1]{gan2012symplectic} we have $\mathcal{S}_\phi \cong A_M^+$. 

\vskip 5pt

\item The pure inner forms of $G$ are of the form $G'={\SO}(V')$, where $V'$ is a quadratic space of the same dimension and discriminant as $V$. If $G=\SO(p,q)$ with $p+q=2n+1$, then these $V'$ are parametrized by indices $(p',q')$ such that
\[
    p'+q'=2n+1, \quad \text{and} \quad p'\equiv p \mod 2.
\]
We have
\[
    \Pi_{\phi} = \bigsqcup_{V'}\Pi_\phi(\SO(V')).
\]

\vskip 5pt

\item By \cite[\S 12]{gan2012symplectic}, there is a unique choice of Whittaker datum and we denote it by $\mathfrak{w}'$.  
\end{itemize}

\vskip 5pt

\subsubsection{\bf The even special orthogonal group}\label{sec even special orthogonal group}
Let $(V,q)$ be a $2n$-dimensional quadratic space and $G={\SO}(V)$. The discriminant $\disc(V)$ is defined as 
\begin{equation}\label{disc V even}
\disc(V)=(-1)^n \det \left(\langle e_i, e_j\rangle\right)_{i,j}\in \BR^\times/\BR^{\times2},
\end{equation}
where $\{e_i\}$ an orthogonal basis of $V$ and $\langle \cdot,\cdot\rangle$ is the symmetric bilinear form associated to $q$ defined in \Cref{symmteric bilinear form}. The discriminant character of $V$ is
\[
\chi_V=\begin{cases}
    \mathbbm{1}\quad & \mbox{if $\disc(V)\in \BR^{\times 2}$,}\\[5pt]
    \sgn\quad & \mbox{otherwise}. 
\end{cases}
\]
Then $\wh{G}=\SO_{2n}(\BC)$ and ${}^L G=\SO_{2n}(\BC)\rtimes \Gal(\BC/\BR)$. The action of $\Gal(\BC/\BR)$ on $\SO_{2n}(\BC)$ is trivial if $\disc(V)\in \BR^{\times 2}$ and is non-trivial otherwise. 

\begin{itemize}
\item An L-parameter $\phi$ of $G$ gives rise to an orthogonal representation $M$ of $\WW_{\BR}$ with $\dim M=2n$ and $\det(M)=\chi_V$. Let $\underline\Phi(G)$ be the set of isomorphism classes of $2n$-dimensional orthogonal representations of $\WW_\BR$ with determinant $\chi_V$, and $\underline\Phi^{\epsilon}(G)$ the subset of $\underline{\Phi}(G)$ consisting of those $M$ containing an $1$-dimensional orthogonal subrepresentation. The map $\phi\mapsto M$ is surjective from $\Phi(G)$ to $\underline\Phi(G)$. The fiber of an orthogonal representation $M\in \underline\Phi(G)$ contains one element if $M\in \underline \Phi^{\epsilon}(G)$ and two elements otherwise. In this paper we only need to consider L-parameters $\phi$ whose associated representation $M \in  \underline \Phi^{\epsilon}(G)$, so in later sections we can still refer to $M$ as an L-parameter of $G$ via the above map, and put $\Pi_{M}(G) = \Pi_{\phi} (G)$ without confusion. 

\vskip 5pt

\item We define the groups $C_M, C_M^+, A_M, A_M^+$ in the same way as in the case of symplectic groups \Cref{sympletic group}. Since $M$ is orthogonal, it admits a decomposition 
\[
   M = P + \sum_i m_iM_i + P^\vee,
\]
where $M_i$ are pairwise distinct irreducible orthogonal subrepresentations of $M$, $m_i$ is the multiplicity of $M_i$ in $M$, and any subrepresentation of $P$ is not orthogonal. Then 
\[
A_M \cong  \prod_i \left(\BZ/2\BZ\right) a_i,
\]
with a canonical basis $\{a_i\}$ corresponding to orthogonal subrepresentations $\{M_i\}$. Depending on whether $M\in\underline \Phi^{\epsilon}(G) $ or not, $A_M^+$ is either an index $2$ subgroup of $A_M$ or $A_M$ itself. According to \cite[Theorem 8.1]{gan2012symplectic} we have $\mathcal{S}_\phi \cong A_M^+$. For each $c \in \mathbb R^\times$, we define a character $\delta_{M,c}$ of $A_M^+$ in the same way as in \eqref{delta}.

\vskip 5pt

\item The pure inner forms of $G$ are of the form $G'={\SO}(V')$, where $V'$ is a quadratic space of the same dimension and discriminant as $V$. If $G=\SO(p,q)$ with $p+q=2n$, then these $V'$ are parametrized by indices $(p',q')$ such that
\[
    p'+q'=2n, \quad \text{and} \quad p'\equiv p \mod 2.
\]
We have
\[
    \Pi_{\phi} = \bigsqcup_{V'}\Pi_\phi(\SO(V')).
\]

\vskip 5pt

\item Let $G^* = \SO(V^+)$ be a quasi-split pure inner form of $G$. By \cite[\S 12]{gan2012symplectic}, the set of Whittaker data for $G^*$ is in bijection with the set of $\SO(V^+)$-orbits of non-isotropic lines $L \subset V^+$ such that $L^\perp$ is split. One can check that two such non-isotropic lines $L$ and $L'$ lie in the same orbit if and only if $\disc(L) = \disc(L')$. If $\disc(L)=c$, we denote by $\mathfrak{w}_c$ the Whittaker datum corresponding to $L$. Note that $\mathfrak{w}_c$ does not depend on the choice of the additive character $\psi$.

\vskip 5pt

\item For two Whittaker data $\Fw_c$ and $\Fw_{c'}$, where $c,c' \in \mathbb R^\times/(\mathbb R^\times)^2$, the discrepancy of the choices of Whittaker data is governed by the following formula (see \cite[Theorem 3.3]{MR3194648}):
\begin{equation}\label{change of Whittater data}
  \mathcal J_{\mathfrak{w}_{c'}}(\pi) = \mathcal J_{\mathfrak{w}_c}(\pi) \otimes \delta_{M, c'/c}, \quad \mbox{for }\pi \in \Pi_{\phi}.
\end{equation}
\end{itemize}

\vskip 5pt

\subsection{Local Langlands correspondence for metaplectic groups}\label{metaplectic group}

Let $W$ be a $2n$-dimensional symplectic space and $G=\Mp(W)$. Following \cite[\S 11]{gan2012symplectic}, we define $\wh{G}=\Sp_{2n}(\BC)$ and ${}^L G=\Sp_{2n}(\BC)\times \Gal(\BC/\BR)$. The concrete descriptions of L-parameters and component groups of metaplectic groups align with those of special odd orthogonal groups. To distinguish them, we shall typically denote an L-parameter for $G=\Mp(W)$ by the symbol $\widetilde \phi$. Following \cite[\S 11]{gan2012symplectic}, we define the LLC for $G=\Mp(W)$ by using the Shimura-Waldspurger correspondence. Recall that the Shimura-Waldspurger correspondence \cite{adams1998genuine}  gives a bijection 
\begin{equation}\label{Shimura}
   \theta_{V,W,\psi}:\Irr_{\gen}(\Mp(W))\longrightarrow \coprod_{\substack{\dim V=2n+1\\
\disc(V)=1}}\Irr(\SO(V)). 
\end{equation}
Combining this bijection with the LLC for special odd orthogonal groups, we obtain a finite to one surjective map
\begin{equation}\label{LLC for metaplectic}
\begin{split}
      \mathcal L_\psi: \Irr_{\gen}(\Mp(W)) & \longrightarrow \Phi(G)\\
     \pi& \longmapsto \mathcal L(\theta_{V,W,\psi}(\pi)).
\end{split}
\end{equation}
Moreover, as the only ``pure inner form'' of the group $G=\Mp(W)$ is itself, we set $\Pi_{\widetilde \phi}=\Pi_{\widetilde \phi}(G)=\mathcal L_\psi^{-1}(\widetilde \phi)$. We also obtain a parametrization of each packet by setting
\begin{equation}\label{LLC for metaplectic character}
\begin{split}
      \mathcal J_\psi: \Pi_{\wt{\phi}} & \longrightarrow \wh{\mathcal{S}_{\wt{\phi}}}\\
     \pi& \longmapsto \mathcal J_{\Fw} (\theta_{V,W,\psi}(\pi)).
\end{split}
\end{equation}

\vskip 10pt

\section{Local Gan-Gross-Prasad Conjecture}\label{section: GGP conjecture}

In this section we recall the statement of the local Gan-Gross-Prasad conjecture for symplectic-metaplectic Fourier-Jacobi models and special orthogonal Bessel models following \cite{gan2012symplectic}. We retain the notation in \Cref{section LLC} in this section. 

\vskip 5pt

\subsection{Symplectic-metaplectic case}\label{section FJ statment}
In this subsection, we state the conjecture for the symplectic-metaplectic case. Let $W$ be a symplectic space over $\BR$. 

\vskip 5pt

\subsubsection{\bf Distinguished character} 
We first recall the distinguished characters defined in \cite[\S 17]{gan2012symplectic}. Let $\wt \phi: \WW_\BR\rightarrow \Sp(M)$ (resp. $\phi: \WW_\BR \rightarrow \SO(N)$) be a tempered L-parameter of $\Mp(W)$ (resp. $\Sp(W)$), where $M$ (resp. $N$) is the symplectic (resp. orthogonal) representation of $\WW_\BR$ associated to $\wt\phi$ (resp. $\phi$). Put $N_1=N\oplus \mathbbm{1}$. We define a character $\eta_{N_1}$ of $C_M$ and a character $\eta_M$ of $C_{N_1}^+$ as follows: for $a\in C_{M}$ and $b \in C^{+}_{N_1}$, we set

\begin{equation}\label{distigulished character}
\begin{cases}
    \eta_{N_1}(a)= \epsilon\left(M^a\otimes N_1,\psi\right)\cdot \det\left(M^a\right)(-1)^{\frac{\dim N_1}{2}}\cdot \det\left(N_1\right)(-1)^{\frac{\dim M^a}{2}};\\[5pt]
    
    \eta_{M} (b) = \epsilon\left(M\otimes N_1^b,\psi\right)\cdot \det\left(M\right)(-1)^{\frac{\dim N_1^b}{2}}\cdot \det\left(N_1^b\right)(-1)^{\frac{\dim M}{2}},
\end{cases} 
\end{equation}
where $M^a$ and $N_1^{b}$ are the $(-1)$-eigenspaces for $a$ in $M$ and $b$ in $N_1$, respectively, and $\epsilon(\cdots)$ is the local root number defined by \eqref{equ: epsilon add} and \eqref{equ: epsilon irr}. 
The characters $\eta_{N_1}$ and $\eta_{M}$ factor through $A_{M}$ and $A^+_{N_1}$. Note that $\mathcal S_{\widetilde \phi}\cong A_M$, and $\mathcal S_\phi\cong A^+_N$ is a subgroup of $A^+_{N_1}$. We shall regard $\eta_{M}$ as a character of $\mathcal S_\phi\cong A^+_N$ by restriction. 

\vskip 5pt

\subsubsection{\bf Statement of the conjecture}

The projection map and the identity map define a diagonal embedding
\[
  \Delta : \Mp(W) \longrightarrow \Sp(W)\times \Mp(W).
\]
Recall that we have fixed a non-trivial additive character $\psi$. This determines a Weil representation $\omega_{\psi}$ of $\Mp(W)$, a Whittaker datum $\mathfrak w$ of $\Sp(W)$ (see \Cref{sympletic group}),
as well as the map $\mathcal{J}_\psi$ in the LLC for the metaplectic group $\Mp(W)$ (see \eqref{LLC for metaplectic character}). The following theorem, known as the symplectic-metaplectic case of the local Gan-Gross-Prasad conjecture, is the main result of this paper.

\begin{thm}\label{GGP FJ}
Let $\wt\phi : \WW_{\BR}\rightarrow \Sp(M)$ be a tempered L-parameter of $\Mp(W)$, and $\phi : \WW_{\BR}\rightarrow \mathrm{SO}(N)$ a tempered L-parameter of $\Sp(W)$. Let $\wt\pi \in \Pi_{\wt\phi}$ and $\pi \in \Pi_{\phi}$. Then
\[
  \Hom_{\Mp(W)}\left(\pi \boxtimes \wt\pi, \omega_{\psi}\right)\neq 0
  \;\Longleftrightarrow\;
  \calJ_\psi\left(\wt\pi\right)\times \calJ_{\mathfrak w}\left(\pi\right)=\eta_{N_1}\times \eta_M,
\]
where $\eta_{N_1}$ and $\eta_M$ are the characters defined in \eqref{distigulished character}.
\end{thm}

\vskip 5pt

\subsection{Special orthogonal case}\label{section bessel case}
In this subsection, we recall the local Gross-Prasad conjecture (i.e. local Gan-Gross-Prasad conjecture for special orthogonal groups) over $\BR$ proved in \cite[Theorem 2.3.2]{chen2022local}. 
Let $V$ be a quadratic space over $\BR$ and $V'$ be a non-degenerate subspace of $V$ of codimension $1$. Assume that both $\SO(V)$ and $\SO(V')$ are quasi-split. Let $V^{\even}$ (resp. $V^{\odd}$) be the unique even (resp. odd) dimensional space in $\{V,V'\}$. 

\vskip 5pt

\subsubsection{\bf Distinguished character} 
We first recall the distinguished characters defined in \cite[\S 17]{gan2012symplectic}. Let $\phi: \WW_{\BR}\rightarrow \mathrm O(M)$ (resp. $\phi': \WW_{\BR}\rightarrow \Sp(N)$) be a tempered L-parameter of $\SO(V^\even)$ (resp. $\SO(V^\odd)$), where $M$ (resp. $N$) is the orthogonal (resp. symplectic) representation of $\WW_\BR$ associated to $\phi$ (resp. $\phi'$). 
We define a character $\eta_{N}$ of $C^+_M$ and a character $\eta_M$ of $C_{N}$ as follows: for  each $a\in C^+_{M}$ and $b \in C_{N}$, we set
\begin{equation}\label{distigulished character Bessel}
\begin{cases}
\eta_N(a)= \epsilon\left(M^a\otimes N,\psi\right)\cdot \det\left(M^a\right)(-1)^{\frac{\dim N}{2}}\cdot \det\left(N\right)(-1)^{\frac{\dim M^a}{2}};\\[5pt]

\eta_M(b)= \epsilon\left(M \otimes N^b,\psi\right)\cdot \det\left(M\right)(-1)^{\frac{\dim N^b}{2}}\cdot \det\left(N^b\right)(-1)^{\frac{\dim M}{2}},
\end{cases}  
\end{equation}
where $M^a$ and $N^b$ are the $(-1)$-eigenspaces for $a$ in $M$ and $b$ in $N$, respectively, and  
$\varepsilon(\cdots)$ is the local root number defined in  \Cref{equ: epsilon add} and \Cref{equ: epsilon irr}. The characters $\eta_N$ and $\eta_M$ factor through $A^+_{M}\cong \mathcal{S}_\phi$ and $A_N\cong \mathcal{S}_{\phi'}$. 

\vskip 5pt

\subsubsection{\bf Statement of Conjecture}
We say a pure inner form $G_\alpha=\SO(V_\alpha)\times \SO(V'_\alpha)$ of $G=\SO(V)\times \SO(V')$ is relevant, if there is an embedding $V'_\alpha \hookrightarrow V_\alpha$ of quadratic spaces, such that $V_\alpha/V'_\alpha\cong V/V'$ as quadratic spaces. If this is the case, then there is a natural embedding $\SO(V'_\alpha)\hookrightarrow \SO(V_\alpha)$. Hence, we have a diagonal embedding 
\[
\Delta: \SO(V'_\alpha)\longrightarrow \SO(V_\alpha)\times \SO(V'_\alpha).
\]
Recall that there is a unique choice of Whittaker datum for $\SO(V^\odd)$, which we denote by $\mathfrak{w}'$. Let $c=\disc(V^\odd)$. Following \cite[\S 17]{gan2012symplectic}, we shall use the Whittaker datum $\mathfrak w_c$ of $\SO(V^\even)$ (see \Cref{sec even special orthogonal group}) to fix the internal parametrization of L-packets in the LLC. The following theorem, known as the local Gan-Gross-Prasad conjecture for special orthogonal groups, has been proved in \cite[Thm. 2.3.2]{chen2022local}. 

\begin{thm}\label{GGP Bessel}
Let $\phi : \WW_{\BR}\rightarrow {\rm O}(M)$ be a tempered L-parameter of $\SO(V^{\even})$, and $\phi': \WW_{\BR}\rightarrow \Sp(N)$ a tempered L-parameter of $\SO(V^{\odd})$. Let $\sigma^{\even}\boxtimes\sigma^{\odd}$ be an irreducible tempered representation of a relevant pure inner form $G_\alpha=\SO(V^{\even}_\alpha)\times \SO(V^{\odd}_\alpha)$, such that $\sigma \in \Pi_{\phi}$ and $\sigma'\in \Pi_{\phi'}$. Then one has 
\begin{equation}\label{equation GGP Bessel}
    \Hom_{\SO(V_\alpha')}\left(\sigma^{\even}\boxtimes \sigma^{\odd}, \BC\right)\neq 0
     \;\Longleftrightarrow\;
    \calJ_{\Fw_c}\big(\sigma^{\even}\big)\times \calJ_{\Fw'}\big(\sigma^{\odd}\big)=\eta_N\times \eta_M,
\end{equation}
where $\eta_{N}$ and $\eta_M$ are the characters defined in \eqref{distigulished character Bessel}, and $c=\disc(V^\odd)$. 
\end{thm}

\vskip 10pt

\section{Preliminaries and Toolkits}\label{sec theta}
In this section, we recall some preliminaries and toolkits for later proofs. 

\vskip 5pt

\subsection{Schwartz induction and Borel's lemma}\label{section: Schwartz}

We recall some basic notions of Schwartz analysis in this subsection. We shall work in the setting of the so-called almost linear Nash groups, i.e., a Nash group admitting a homomorphism to some general linear group with finite kernel. All groups we are working with in this paper (orthogonal, symplectic, metaplectic and Jacobi groups, as well as their closed subgroups) belong to this class of groups. Let $G$ be an almost linear Nash group, we denote by $\Smod(G)$ the category of Fr\'echet representations of moderate growth of $G$.

\vskip 5pt

Let $H$ be a Nash subgroup of $G$, and $(\pi_H, V)\in \Smod(H)$. Then there is a continuous map
\[
    I: \mathcal{S}(G,V) \longrightarrow C^\infty(G,V), \quad  f\longmapsto \left(g\mapsto \int_H \pi_H(h)f\left(h^{-1}g\right)\,dh\right)
\]
from the space of Schwartz functions $\mathcal{S}(G,V)$ to the space of smooth functions $C^\infty(G,V)$, with the image landing in the usual un-normalized induction $\unInd_H^G V$. The image of $I$, equipped with the quotient topology of $\mathcal{S}(G,V)$, is called the Schwartz induction of $\pi_H$ and will be denoted by $\unind_H^G V$. We also define the \emph{normalized Schwartz induction} of $V$ from $H$ to $G$ to be
\begin{equation}\label{normalized Schwartz induction}
       \ind_H^G V \;:=\; \left(\,\unind_H^G\!\left(V\cdot \delta_H^{1/2}\right)\right)\cdot \delta_G^{-1/2}.
\end{equation}
Here $\delta_H$ and $\delta_G$ are the modulus characters of $H$ and $G$. In particular, if $G$ is a reductive group and $P$ is a parabolic subgroup of $G$, then $\ind_P^G V=\unind_P^G\left(V\cdot \delta_P^{1/2}\right)$ coincides with the usual (normalized) parabolic induction. 

\vskip 5pt

For a representation $\left(\pi, V\right)\in\Smod\left(G\right)$, we put
\[
    V_G = V\big/\left\langle \pi\left(g\right)v-v~\big|~g\in G,~v\in V\right\rangle,
\]
equipped with the quotient topology. It might not be Hausdorff. By definition, we have 
\[
    V_G^* \cong \Hom_G\left(V, \mathbbm{1}_G\right),
\]
where $V_G^*$ is the continuous linear dual of $V_G$, and $\mathbbm{1}_G$ is the trivial representation of $G$. The next proposition is a version of the Frobenius reciprocity in the Schwartz analysis setting. Readers may consult \cite[Proposition 6.6, Theorem 6.8]{MR4211018} for more details.

\begin{pro}\label{Frobenious}
Let $H$ be a Nash subgroup of $G$ and  $(\pi,V)\in \Smod(H)$. Then 
\[
   \left(\left(\unind_H^G \left( V\otimes  \delta_H\right)\right)\otimes\delta_G^{-1}\right)_G \simeq V_H.
\]
\end{pro}

\vskip 5pt

\begin{cor}\label{Frobenious cor}
Let $G$ be a reductive Nash group, and $H$ a reductive Nash subgroup of $G$. Let $(\pi,V)$ be a Casselman-Wallach representation of $H$, and $(\sigma,W)$ a Casselman-Wallach representation of $G$. Then we have 
\[
    \Hom_G\left(\unind_H^G ( V\otimes \delta_H), W\otimes \delta_G\right)\cong \Hom_H\left(W^\vee\big|_H, V^\vee\right), 
\]
where $V^\vee$ stands for the contragredient representation of $V$, and likewise for $W^\vee$.
\end{cor}

\vskip 5pt

\begin{proof}
    Since $(\pi,V)$ and $(\sigma,W)$ are Casselman-Wallach representations, it follows from \cite[Lemma 2.2.22]{chen2023multiplicity} that we have
    \[
    \begin{aligned}
        \Hom_G\left(\unind_H^G ( V\otimes \delta_H), W\otimes \delta_G\right)&\cong \Hom_G\left(\unind_H^G (V\otimes \delta_H)\otimes W^\vee \otimes \delta_G^{-1},\mathbbm{1}_G\right)\\
        &\cong \, \Hom_{G}\left(\unind_H^G \left(\left(V\otimes W^\vee\big|_H\right)\otimes \delta_H\right) \otimes \delta_G^{-1},\mathbbm{1}_G\right)
    \end{aligned}
    \]
    Then by \Cref{Frobenious}, we have
    \[
    \begin{aligned}
        \Hom_{G}\left(\unind_H^G \left(\left(V\otimes W^\vee\big|_H\right)\otimes \delta_H\right) \otimes \delta_G^{-1},\mathbbm{1}_G\right) &\cong \, \Hom_H\left(V\otimes W^\vee\big|_H,\mathbbm{1}_H\right)\\
        &\cong \, \Hom_H\left(W^\vee\big|_H, V^\vee\right).
    \end{aligned}
    \]
\end{proof}

\vskip 5pt

In later proofs we also need to do some orbit analysis. We now introduce a version of Borel's lemma, following \cite[Proposition 8.2, 8.3]{MR4211018} and \cite[Proposition 2.5]{xue2020bessel}. Let $\mathcal{X}$ be a Nash manifold and $\mathcal{Z}$ a closed Nash submanifold. We put $\mathcal{U} = \mathcal{X} \backslash \mathcal{Z}$. Let $\mathcal{E}$ be a tempered vector bundle over $\mathcal{X}$, then the ``extension by zero'' of Schwartz sections over $\mathcal{U}$ gives us the following exact sequence:
\[
    0 \longrightarrow \Gamma^\mathcal{S}\left(\mathcal{U}, \mathcal{E}\right) \longrightarrow \Gamma^\mathcal{S}\left(\mathcal{X}, \mathcal{E}\right) \longrightarrow \Gamma^\mathcal{S}_{\mathcal{Z}}\left(\mathcal{X}, \mathcal{E}\right) \longrightarrow 0,
\]
where the symbol $\Gamma^\mathcal{S}$ stands for taking Schwartz sections, and
\[
    \Gamma^\mathcal{S}_{\mathcal{Z}}\left(\mathcal{X}, \mathcal{E}\right) := \Gamma^\mathcal{S}\left(\mathcal{X}, \mathcal{E}\right)\big/\Gamma^\mathcal{S}\left(\mathcal{U}, \mathcal{E}\right).
\]
Let $N_{\mathcal{Z}|\mathcal{X}}^*$ be the conormal bundle of $\mathcal{Z}$ and $\mathcal{N}_{\mathcal{Z}|\mathcal{X}}^*$ the complexification of $N_{\mathcal{Z}|\mathcal{X}}^*$.

\begin{pro}\label{P:BorelLemma}
There is a descending filtration on $\Gamma^\mathcal{S}_{\mathcal{Z}}\left(\mathcal{X}, \mathcal{E}\right)$ indexed by $k\in\mathbb{N}$
\[
    \Gamma^\mathcal{S}_{\mathcal{Z}}\left(\mathcal{X}, \mathcal{E}\right)=\Gamma^\mathcal{S}_{\mathcal{Z}}\left(\mathcal{X}, \mathcal{E}\right)_0 \supset \Gamma^\mathcal{S}_{\mathcal{Z}}\left(\mathcal{X}, \mathcal{E}\right)_1 \supset \Gamma^\mathcal{S}_{\mathcal{Z}}\left(\mathcal{X}, \mathcal{E}\right)_2 \supset \cdots,
\]
such that
\[
    \Gamma^\mathcal{S}_{\mathcal{Z}}\left(\mathcal{X}, \mathcal{E}\right) = \lim_{\longleftarrow} \Gamma^\mathcal{S}_{\mathcal{Z}}\left(\mathcal{X}, \mathcal{E}\right)\big/\Gamma^\mathcal{S}_{\mathcal{Z}}\left(\mathcal{X}, \mathcal{E}\right)_k,
\]
with each graded pieces are isomorphic to
\[
    \Gamma^\mathcal{S}\left(\mathcal{Z}, \Sym^k\mathcal{N}_{\mathcal{Z}|\mathcal{X}}^*\otimes\mathcal{E}\big|_\mathcal{Z}\right), \quad k\in\mathbb{N}.
\]
Moreover, if $\mathcal{X}$ is a $G$-Nash manifold, $\mathcal{Z}$ is stable under the $G$-action, and $\mathcal{E}$ is a tempered $G$-bundle, then this filtration is stable under the $G$-action.

\end{pro}

\vskip 5pt

\subsection{Representation theory of Jacobi groups}\label{Section Weil representation}

In this subsection, we recall some preliminaries on Weil representations and the representation theory of Jacobi groups. 

\subsubsection{\bf The Weil representation}\label{SubSection Weil representation}

Let $W$ be a symplectic vector space over $\BR$. We denote by 
\[
\Hei(W) := \BG_a\oplus W 
\] 
the Heisenberg group associated with $W$, with multiplication law 
\[
\left(z_1,w_1\right)+\left(z_2,w_2\right)=\left(z_1+z_2+\frac{1}{2}\left\langle w_1,w_2\right\rangle, w_1+w_2\right). 
\]
The symplectic group $\Sp(W)$ acts naturally on $\Hei(W)$ by 
\begin{equation}\label{equ: adj of Jacobi}
g \cdot \left(z,w\right) = \left(z, g \cdot w\right), \quad \mbox{for } z \in \BG_a,\; w \in W,\; g \in \Sp(W).
\end{equation}
We define the \emph{metaplectic Jacobi group} to be
\[
\wt{\J}(W) := \Mp(W) \ltimes \Hei(W),
\]
where the action of $\Mp(W)$ on $\Hei(W)$ factors through the natural action of $\Sp(W)$ on $\Hei(W)$ given by \Cref{equ: adj of Jacobi}. For a non-trivial additive character $\psi$ of $\mathbb R$, by the well-known Stone-von Neumann theorem, there is a unique irreducible representation $\omega_\psi$ of $\Hei(W)$ with central character $\psi$, called the \emph{Weil representation}. This representation $\omega_\psi$ uniquely extends to $\wt\J(W)$. By a slight abuse of notation, we use the same symbol to denote the restriction of $\omega_\psi$ to $\Mp(W)$, which is the one appearing in the main theorem.

\vskip 5pt

\subsubsection{\bf Fourier-Jacobi functors}\label{Rep of Jacobi}

In this section, we review the representation theory of Jacobi groups and reinterpret the Hom space in \eqref{fj MODEL} in terms of it. The point is that, in terms of Jacobi groups, this Hom space becomes completely parallel to the Bessel case recalled in \Cref{section bessel case}.

\vskip 5pt

We retain the notation in \Cref{SubSection Weil representation}. The \emph{Jacobi group} associated with $W$ is defined as
\[
\J(W) := \Sp(W) \ltimes \Hei(W),
\]
where the action of $\Sp(W)$ on $\Hei(W)$ is given by \Cref{equ: adj of Jacobi}. We denote by $\Smod_{\psi}(\J(W))$ the category of smooth Fr\'echet representations of moderate growth of $\J(W)$ with central character $\psi$, and by $\Smod_{\gen}(\Mp(W))$ the category of genuine smooth Fr\'echet representations of moderate growth of $\Mp(W)$. There are two functors: 
\begin{equation}\label{Functor F}
\begin{split}
\mathcal F_\psi: \Smod_{\gen}(\Mp(W))&\longrightarrow \Smod_{\psi}(\J(W))\\
\pi &\longmapsto \pi\otimes \omega_\psi
\end{split}
\end{equation}
and 
\begin{equation}\label{Functor G}
    \begin{split}
\mathcal G_\psi: \Smod_{\psi}(\J(W))&\longrightarrow \Smod_{\gen}(\Mp(W))\\
\sigma &\longmapsto \Hom_{\Hei(W)}\left(\omega_\psi, \sigma\right).
\end{split}
\end{equation}
These functors are called \emph{Fourier-Jacobi functors}. By \cite[Proposition 4.2]{MR2891317}, we have:

\begin{pro}\label{Sun thm}
These two functors $\mathcal{F}_\psi$ and $\mathcal{G}_\psi$ are mutually inverse equivalences of categories between 
$\Smod_{\gen}(\Mp(W))$ and $\Smod_{\psi}(\J(W))$.   
\end{pro}

\vskip 5pt

Recall that in the Fourier-Jacobi model (see \eqref{fj MODEL}), one studies the Hom space 
\begin{equation}\label{fj MODEL 5}
 \Hom_{H}\!\left(\pi \boxtimes \wt{\pi}, \, \omega_\psi\right) \cong \Hom_{H}\!\left(\pi \boxtimes (\wt{\pi}\otimes \overline{\omega}_{\psi}), \, \mathbbm{1}_H\right),
\end{equation}
where $H=\Delta\Mp(W)$ is the diagonal embedding of $\Mp(W)$ into $G=\Sp(W)\times \Mp(W)$, $\pi \in \Irr(\Sp(W))$ and $\wt{\pi} \in \Irr_{\gen}(\Mp(W))$. Note that by \Cref{Sun thm}, the tensor product $\wt{\pi} \otimes \overline{\omega}_{\psi} \cong \mathcal{F}_{\bar\psi}\left(\wt\pi\right)$ can be regarded as an irreducible representation of $\J(W)$.
Therefore, if we set 
\[
G' = \Sp(W) \times \J(W) \quad \mbox{and} \quad H' = \Delta\Sp(W),
\]
where $\Delta\Sp(W)$ denotes the diagonal embedding of $\Sp(W)$ into $G'$, then the right hand side of \Cref{fj MODEL 5} can be rewritten as
\begin{equation}\label{fj MODEL rewrite}
\Hom_{H'}\!\left(\pi \boxtimes \mathcal{F}_{\bar\psi}\left(\wt\pi\right), \, \mathbbm{1}_H\right).
\end{equation}
Comparing the Hom space in \Cref{fj MODEL rewrite} with the Bessel model \eqref{equation GGP Bessel}, one can heuristically regard the Jacobi group $\J(W)$ as a substitution of $\SO(V)$ in the Fourier-Jacobi model.
Alternatively, one can also consider corank two Fourier-Jacobi models as follows. Let $W^+ = W \oplus \mathcal{H}'$, where $\mathcal{H}'$ is the (symplectic) hyperbolic plane with a standard basis $\left\{e,f\right\}$. We write elements in $\Sp\left(W^+\right)$ as block matrices using the decomposition $W^+=\langle e\rangle \oplus W\oplus \langle f\rangle$. Then we can embed $\J(W)$ into $\Sp\left(W^+\right)$ by setting 
\begin{equation*}
       g \mapsto \begin{pmatrix}
       1 &  & \\
        & g & \\
        &  & 1
        \end{pmatrix} \quad \mbox{and} \quad 
(z, w) \mapsto \begin{pmatrix}
    1 & w^*_e & z\\
        & 1 &w\\
        &  & 1 
\end{pmatrix}, \quad \mbox{for } g\in\Sp(W)\mbox{ and }(z,w)\in \Hei(W). 
\end{equation*}
Here $w^*_e\in \Hom(W,\langle e \rangle)$ is defined by $w^*_e(v)= \langle w, v \rangle \cdot e$ for $v\in W$.
One can easily check that 
\begin{equation}\label{embedding J to Sp}
    \J(W) \cong  \Stab_{\Sp\left(W^+\right)}\left(e\right)
\end{equation}
via this embedding. Set 
\[
G''=\J\left(W\right)\times \Sp\left(W^+\right) \quad\mbox{and}\quad H''=\Delta\J\left(W\right),
\]
where $\Delta\J\left(W\right)$ denotes the diagonal embedding of $\J\left(W\right)$ into $G''$. Let $\wt \pi\in\Irr_{\gen}\left(\Mp\left(W\right)\right)$ and $\pi^+ \in \Irr\left(\Sp\left(W^+\right)\right)$. Then, the corank two Fourier-Jacobi model defined in \cite[\S 12]{gan2012symplectic} can be written as
\begin{equation}\label{fj MODEL rewrite II}
\Hom_{H''}\left(\mathcal{F}_{\bar\psi}\left(\wt\pi\right)\boxtimes \pi^+, \, \mathbbm{1}_{H''}\right).
\end{equation}
This is again an analogy of the Bessel model; in this case the Jacobi group $\J\left(W\right)$ plays the role of $\SO\left(V'\right)$. We remark that \Cref{conj: GP in introduction} for above Hom space can be reduced to the case of \eqref{fj MODEL rewrite}. We refer the readers to \cite{chen2023multiplicity} for this reduction step. 

\vskip 5pt

Next we define the notion of ``parabolic induction'' for Jacobi groups, and show that the functors defined in \eqref{Functor F} and \eqref{Functor G} are compatible with ``parabolic inductions" on both sides. We should remark here that ``parabolic subgroups'' for Jacobi groups defined below are \emph{not} the usual notion of parabolic subgroups.

\vskip 5pt

We first define the notion of ``parabolic subgroups''. Consider the case of ``maximal parabolics''. Recall that for classical groups, maximal parabolic subgroups are stabilizers of points in Grassmannians; thus, we need to find suitable substitutions for Grassmannians. For each $0\leq k\leq \frac{1}{2}\dim W$, let $\mathcal{F}_{W,k}$ be the flag variety of $\Sp(W)$ consisting of the $k$-dimensional isotropic subspaces of $W$. Consider the dual tautological bundle 
\begin{equation}\label{flag J}
    \mathcal{F}^{\J}_{W,k}\coloneqq \left\{\left(X,v^*\right)\,\big|\, X\in \mathcal{F}_{W,k}, v^*\in X^*\right\}
\end{equation}
of $\mathcal{F}_{W,k}$.
There is a natural transitive action of $\J(W)$ on $\mathcal{F}^{\J}_{W,k}$ given by
\begin{align*}
    g \cdot \left(X,v^*\right) &= \left(g\cdot X, g \cdot v^*\right),   &  \mbox{for } g&\in \Sp\left(W\right);\\[5pt]
    \left(z,w\right) \cdot \left(X,v^*\right) &= \left(X, v^*+ \iota_X\left(w\right)\right), &  \mbox{for } \left(z,w\right)&\in \Hei\left(W\right).
\end{align*}
Here $X^*$ is the linear dual of $X$; $g\cdot v^*$ is a linear functional on $g\cdot X$ defined by the formula $\left(g\cdot v^*\right)\left(v\right) = v^*\left(g^{-1}\cdot v\right)$ for $v\in g\cdot X$; and $\iota_X: W \rightarrow W/X^\perp\cong X^*$ is the natural porjection map.
For each point $\left(X,v^*\right)\in \mathcal{F}^{\J}_{W,k}$, we define a corresponding ``maximal parabolic subgroup''
\[
  P_{X,v^*} \coloneqq \Stab_{\J(W)}\left(X,v^*\right).
\]
All such subgroups $P_{X,v^*}$ are conjugate as $\left(X,v^*\right)$ varies in $\mathcal{F}^{\J}_{W,k}$. To describe them explicitly, we choose a point $\left(X,0\right)\in\mathcal{F}^{\J}_{W,k}$.
Let $W_0 = X^\perp/X$, equipped with the symplectic form induced from $W$. Then
\[
  P_{X,0} \cong Q_{X}\rtimes \Hei\left(X^\perp\right),
\]
where $Q_{X}$ is the maximal parabolic subgroup of $\Sp\left(W\right)$ stabilizing the isotropic subspace $X$, and for each subspace $Y\subset W$, possibly degenerate, we set $\Hei\left(Y\right)=\mathbb{G}_a\oplus Y$, viewed as a subgroup of $\Hei\left(W\right)$. There is a natural projection
\[
  P_{X,0} \twoheadrightarrow M_{X,0} \coloneqq \GL\left(X\right)\times \J\left(W_0\right),
\]
and we refer to $M_{X,0}$ as the ``Levi quotient'' of $P_{X,0}$. To simplify notations, when there is no confusion, we shall suppress ``$0$'' from the subscript and write
\begin{equation*}
  P_{X,0} = P_{X}\quad\mbox{and}\quad M_{X,0} = M_{X}.
\end{equation*}
For each $\left(X,v^*\right)\in \mathcal{F}^{\J}_{W,k}$, put
\[
  X_{v^*} \coloneqq \left\{ x + v^*(x)e \,\big|\, x\in X \right\} \subseteq W^+.
\]
One checks that $X_{v^*}$ is an isotropic subspace of $W^+$ with $\dim X_{v^*} = k$. This gives a $\J(W)$-equivariant embedding
\begin{equation*}
  \begin{split}
    i_k : \mathcal{F}^{\J}_{W,k} &\longrightarrow \mathcal{F}_{W^+,k},\\
    \left(X,v^*\right) &\longmapsto X_{v^*},
  \end{split}
\end{equation*}
where $\J(W)$ acts on $\mathcal{F}_{W^+,k}$ through the isomorphism \eqref{embedding J to Sp}. The image of $i_k$ can be described as
\[
  i_k\left(\mathcal{F}^{\J}_{W,k}\right) = \left\{ Y \in \mathcal{F}_{W^+,k} \,\big|\, Y \subseteq \left\langle e\right\rangle^\perp\mbox{ and }\left\langle e\right\rangle \nsubseteq Y \right\}.
\]
There is also a natural $\J(W)$-equivariant surjection
\begin{equation*}
  \begin{split}
    p_k : \mathcal{F}^{\J}_{W,k} &\longrightarrow \mathcal{F}_{W,k},\\
    \left(X,v^*\right) &\longmapsto X,
  \end{split}
\end{equation*}
where the action of $\J(W)$ on $\mathcal{F}_{W,k}$ factors through the quotient map $p:\J(W) \twoheadrightarrow \Sp(W)$.
A point $\left(X,v^*\right)\in \mathcal{F}^{\J}_{W,k}$ specifies maximal parabolic subgroups $Q_{X_{v^*}}$ and $Q_{X}$ of $\Sp\left(W^+\right)$ and $\Sp\left(W\right)$, stabilizing isotropic subspaces $X_{v^*}$ and $X$ respectively. By using maps $i_k$ and $p_k$ defined above, one deduces that 
\begin{equation}\label{Parabolic compatible}
    P_{X,v^*}= Q_{X_{v^*}}\cap \Sp\left(W^+\right) \quad \mbox{and} \quad p\left(P_{X,v^*}\right) = Q_{X}.
\end{equation}
These discussions describe ``maximal parabolic subgroups'' of Jacobi groups. We shall also call ``maximal parabolic subgroups'' \emph{corank}-$1$ ``parabolic subgroups''.

\vskip 5pt

General ``parabolic subgroups'' of Jacobi groups can be defined similarly: consider a flag $\underline{X}$ of isotropic subspaces
\[
  X_{1} \subseteq X_{2} \subseteq \cdots \subseteq X_{r} \subseteq W,
\]
and a vector $v^* \in X_{r}^*$. Let $P_{\underline{X},v^*}$ be the ``parabolic subgroup’’ of $\J(W)$ stabilizing this flag $\underline{X}$ and $v^*$. The group $\J(W)$ acts transitively on such pairs of flags and dual vectors, hence all stabilizers are conjugate and isomorphic. We fix a representative $\left(\underline{X},0\right)$ and set $P_{\underline{X}}=P_{\underline{X},0}$ to simplify the notation. Then we have
\[
  P_{\underline{X}} \cong Q_{\underline{X}} \rtimes \Hei\left(X_{r}^\perp\right),
\]
where $Q_{\underline{X}}$ is the parabolic subgroup of $\Sp(W)$ stabilizing the above flag $\underline{X}$, and the ``Levi quotient'' of $P_{\underline{X}}$ is defined to be
\[
  M_{\underline{X}} \cong \GL\left(X_1\right)\times\GL\left(X_2/X_1\right)\times \cdots \times \GL\left(X_r/X_{r-1}\right)\times \J\left(W_0\right),
\]
with $W_0 \cong X_{r}^\perp/X_{r}$.

\vskip 5pt

Next we define the notion of ``parabolic induction'' for Jacobi groups. In the case of ``maximal parabolics'', let $X\in\mathcal{F}_{W,k}$ be an isotropic subspace of $W$, and $P_{X} = P_{X,0}$ the ``maximal parabolic subgroup'' of $\J(W)$ associated with $(X,0)\in\mathcal{F}_{W,k}^\J$. For any $\tau\in \Smod(\GL(X))$ and $\pi_0^\J\in\Smod_{\psi}(\J(W_0))$, where $W_0 = X^\perp/X$, we define the ``normalized parabolic induction''
\begin{equation}\label{normalized Schwartz parabolic induction Jacobi}
\begin{split}
     \tau\rtimes \pi_0^\J\coloneqq & \,\ind_{P_{X}}^{\J\left(W\right)}\left(\tau\boxtimes \pi_0^\J\right)\\
     =&\,\unind_{P_{X}}^{\J\left(W\right)}\left(\left(\tau\boxtimes \pi_0^\J\right)\otimes \delta_{P_{X} }^{\frac{1}{2}}\right).
\end{split}
\end{equation}
Here $\tau\boxtimes \pi_0^\J$ is a representation of the ``Levi quotient'' $M_X = \GL(X)\times\J(W_0)$, and we also regard it as a representation of $P_{X}$ by inflation. The modulus character $\delta_{P_{X}}$  factors through the projection to $\GL(X)$ and is given by 
\begin{equation}\label{modular character Q_J}
    \delta_{P_{X} }(a)= |\det a|^{\dim W_0+k+2},\quad\mbox{for } a\in \GL(X). 
\end{equation}
Similarly, for a general ``parabolic subgroup'' $P_{\underline{X}}$ of $\J(W)$, where $\underline{X}$ is a flag of isotropic subspaces of $W$ as before, and for any $\tau_i\in \Smod(\GL(X_i/X_{i-1}))$ and $\pi_0^\J \in \Smod_{\psi}\left(\J(W_0)\right)$, we define the normalized parabolic induction to be
\[
  \tau_1\times \cdots \times \tau_r\rtimes \pi_0^\J \coloneqq \ind_{P_{\underline{X}}}^{\J\left(W\right)} \left(\tau_1\boxtimes \cdots \boxtimes \tau_r \boxtimes \pi_0^\J\right).
\]

\vskip 5pt

On the other hand, for metaplectic groups we already have the notions of parabolic subgroups and inductions. Parabolic subgroups of $\Mp(W)$ are precisely preimages of those of $\Sp(W)$. Let $X$ be an isotropic subspace of $W$, $Q_X$ the parabolic subgroup of $\Sp(W)$ stabilizing $X$, and $\wt{Q}_{X}$ the preimage of $Q_{X}$ in $\Mp(W)$. We let $U_{X}$ be the unipotent radical of $\wt{Q}_X$ and $\wt{M}_{X}$ the Levi quotient. Then we have 
\[
\wt{M}_{X}\cong \wt{\GL}\left(X\right) \times _{\mu_2} \Mp\left(W_0\right),
\]
where $\wt{\GL}\left(X\right)$ is a double cover of ${\GL}\left(X\right)$, with the underlying set 
\[
\wt{\GL}\left(X\right)= {\GL}\left(X\right) \times \{ \pm 1\}
\]
and the multiplication law 
\[
\left(g_1,\epsilon_1\right)\cdot \left(g_2,\epsilon_2\right)= \left(g_1g_2, \epsilon_1\epsilon_2\cdot \left(\det g_1,\det g_2\right)_\BR\right).
\]
Here $\left(\cdot,\cdot\right)_\BR$ denotes the Hilbert symbol. According to \cite[\S 1]{adams1998genuine}, there is a genuine character $\chi_\psi$ of $\wt{\GL}\left(X\right)$. By tensoring with this genuine character $\chi_\psi$, one obtains a bijection between the set of equivalence classes of representations of ${\GL}\left(X\right)$ and the set of equivalence classes of genuine representations of $\wt{\GL}\left(X\right)$.
Let $\tau\in\Smod\left({\GL}\left(X\right)\right)$ and $\wt\pi_0\in\Smod_{\gen}\left(\Mp\left(W_0\right)\right)$. We can form the normalized parabolic induction
\[
\begin{split}
\tau \rtimes \wt\pi_0 \coloneqq &\,\ind_{\wt{Q}_{X}}^{\Mp(W)}\left(\tau\chi_\psi\boxtimes \wt\pi_0\right)\\
=&\,\unind_{\wt{Q}_{X}}^{\Mp(W)}\left(\left(\tau\chi_\psi\boxtimes \wt\pi_0\right)\otimes \delta_{\wt{Q}_{X}}^{\frac{1}{2}}\right).
\end{split}
\]
Here the modulus character $\delta_{\wt{Q}_{X}}$ factors through the projection to $\GL(X)$, and is given by 
\begin{equation}\label{modular character Q}
    \delta_{\wt{Q}_{X}}\left(a\right)= |\det a|^{\dim W_0+k+1},\quad\mbox{for } a\in \GL\left(X\right). 
\end{equation}
Similarly, for a general parabolic subgroup $\wt{Q}_{\underline{X}}$ of $\Mp(W)$, where $\underline{X}$ is a flag of isotropic subspaces of $W$ as before, and for any $\tau_i\in \Smod(\GL(X_i/X_{i-1}))$ and $\wt\pi_0 \in \Smod_{\gen}\left(\Mp(W_0)\right)$, we define the normalized parabolic induction to be
\[
  \tau_1\times \cdots \times \tau_r\rtimes \wt\pi_0 \coloneqq \ind_{Q_{\underline{X}}}^{\Mp\left(W\right)} \left(\tau_1\chi_\psi\boxtimes \cdots \boxtimes \tau_r\chi_\psi \boxtimes \wt\pi_0\right).
\]

\begin{pro}\label{F_k and G_k vs parabolic}
The functors $\mathcal{F}_\psi$ and $\mathcal{G}_\psi$ defined in \Cref{Functor F} and \Cref{Functor G} are compatible with ``parabolic induction functors" on both sides. i.e., we have 
\begin{equation}\label{F_k}
  \mathcal{F}_\psi\left(\tau_1\times \cdots \times \tau_r \rtimes \wt\pi_0\right) \cong \tau_1\times \cdots \times \tau_r\rtimes \mathcal{F}_\psi\left(\wt\pi_0\right)
\end{equation}
and likewise
\begin{equation}\label{G_k}
\mathcal{G}_\psi\left(\tau_1\times \cdots \times \tau_r \rtimes \pi_0^\J\right) \cong \tau_1\times \cdots \times \tau_r\rtimes \mathcal{G}_\psi\left(\pi_0^\J\right).
\end{equation}
\end{pro}

\vskip 5pt

\begin{proof}
In the case of ``maximal parabolic subgroups'', the desired conclusion is shown during the proof of \cite[Proposition 4.3]{liu2013uniqueness}. For general ``parabolic subgroups'', the desired conclusion follows from induction in stages.

\end{proof}

\vskip 5pt

As a direct consequence of \Cref{Sun thm} and \Cref{F_k and G_k vs parabolic}, we obtain a Langlands classification of irreducible representations of Jacobi groups. Let $\pi^\J\in\Smod_\psi(\J(W))$ be an irreducible representation. Then there exists a ``parabolic induction'' 
\[
\tau_1|\det|^{a_1}\times\cdots\times \tau_r|\det|^{a_r}\rtimes\pi_0^\J,
\]
such that $\pi^\J$ is the unique irreducible quotient of it. Here for each $1\leq i\leq r$, $\tau_i$ is either a unitary character of $\GL_1(\BR)$ or a discrete series of $\GL_2(\BR)$, $a_i$ is a real number, such that $ a_1 \geq \cdots \geq a_r>0$; $\pi_0^\J$ is an irreducible representation of some appropriate Jacobi group, such that $\mathcal{G}_\psi\left(\pi_0^\J\right)$ is tempered. We denote by 
\begin{equation*}
        \pi = LQ\left(\tau_1|\det|^{a_1}\times\cdots\times \tau_r|\det|^{a_r}\rtimes\pi_0^\J\right),
\end{equation*}
and call $\pi$ the ``Langlands quotient'' of $\tau_1|\det|^{a_1}\times\cdots\times \tau_r|\det|^{a_r}\rtimes\pi_0^\J$. In the case that $\mathcal{F}_\psi(\pi^\J)$ is tempered, we say that $\pi^\J$ is tempered.

\vskip 5pt

\subsection{MVW-involution}\label{sec MVW}

Let $V$ be a non-degenerate quadratic or symplectic space over $\mathbb R$, and let $G = G(V)$ be the isometry group associated with $V$ (or the metaplectic group if $V$ is symplectic). There exists a remarkable covariant exact functor on $\Smod(G)$ (or $\Smod_{\gen}(G)$ if $G$ is a metaplectic group), called the \emph{MVW-involution}:

\begin{thm}\label{thm MVW}
There is a covariant exact functor
\[
    \MVW : \Smod(G) \longrightarrow \Smod(G) \quad \left(\mbox{or }\Smod_{\gen}(G) \longrightarrow \Smod_{\gen}(G)\mbox{ if $G$ is metaplectic}\right),
\]
sending each (genuine) representation $\pi$ to another $\pi^\MVW$, satisfies the following properties:
\begin{enumerate}
    \item $\MVW$ is involutive, i.e., $\left(\pi^\MVW\right)^\MVW \cong \pi$ for all $\pi \in \Smod(G)$ (or $\Smod_{\gen}(G)$ if $G$ is a metaplectic group);

    \vskip 5pt
    
    \item if $\pi \in \Irr(G)$ (or $\Irr_{\gen}(G)$ if $G$ is a metaplectic group), then $\pi^\MVW \cong \pi^\vee$;

    \vskip 5pt
    
    \item for a parabolic induced representation $\tau \rtimes \pi_0$, one has
      \[
        \left(\tau \rtimes \pi_0\right)^\MVW \cong \tau  \rtimes \pi_0^\MVW.
      \]
\end{enumerate}
\end{thm}

\vskip 5pt

\begin{proof}
  See \cite[Chapter 4.II.1]{MR1041060}.
  
\end{proof}

\vskip 5pt

\begin{rmk}
By \Cref{Sun thm}, we may transport the MVW-functor to $\Smod_{\psi}(\J(W))$ via the category equivalence with $\Smod_{\gen}(\Mp(W))$. In particular, the resulting functor on $\Smod_{\psi}(\J(W))$ satisfies the same properties as in \Cref{thm MVW}.
\end{rmk}

\vskip 5pt

\subsection{Local theta lifts}
In this subsection, we recall some facts about theta lifts for symplectic-orthogonal reductive dual pairs. 

\vskip 5pt

\subsubsection{\bf Theta lifting and Howe duality}\label{localthteta}

We fix a non-trivial additive character $\psi$ of $\BR$. Let $V$ and $W$ be a quadratic and a symplectic space over $\BR$, respectively. Then as explicated in \cite{MR1286835}, one can associate a Weil representation $\omega=\omega_{V,W,\psi}$ to the group
\[
\begin{cases}
\Mp(W)\times \Or(V), \quad &\text{if $\dim V$ is odd};\\[5pt]

\Sp(W)\times \Or(V), \quad &\text{if $\dim V$ is even}.
\end{cases}
\]
Given $\pi\in \Irr_{\gen}(\Mp(W))$ or $\pi\in\Irr(\Sp(W))$, one can consider the maximal $\pi$-isotypic quotient 
\[
\pi \boxtimes \Theta_{V,W,\psi}(\pi)
\]
of $\omega$. Here $\Theta_{V,W,\psi}(\pi)$ is the multiplicity space, which is indeed a finite length representation of $\Or(V)$. The celebrated Howe duality conjecture, which was proved by Waldspurger \cite{MR1159105} and Gan-Takeda \cite{MR3454380}, asserts that:
\begin{itemize}
    \item the maximal semi-simple quotient $\theta_{V,W,\psi}(\pi)$ of $\Theta_{V,W,\psi}(\pi)$ is either zero or irreducible;
    
    \vskip 5pt

    \item if $\pi_1\not\simeq \pi_2$ are two non-isomorphic irreducible smooth representations such that both $\theta_{V,W,\psi}(\pi_1)$ and $\theta_{V,W,\psi}(\pi_2)$ are nonzero, then $\theta_{V,W,\psi}(\pi_1)\not\simeq \theta_{V,W,\psi}(\pi_2)$. 
\end{itemize}
\vskip 5pt
Similarly, given $\sigma \in \Irr(\Or(V))$, we obtain smooth finite length representations $\Theta_{W,V, \psi}(\sigma)$ and $\theta_{W, V, \psi}(\sigma)$ of $\Sp(W)$ or $\Mp(W)$. 

\vskip 5pt

\subsubsection{\bf (Almost) equal rank theta lifts and Prasad's conjecture}\label{Prasad conjecture}

To prove our main result using a seesaw argument, we need to describe the local theta lift for even orthogonal-symplectic dual pairs in terms of the LLC. In the case that 
\[
    |\dim V - \dim W| \leq 2,
\]
such descriptions has been provided by the so-called Prasad's conjecture (now a theorem due to \cite{adams1998genuine} \cite{MR1403972} \cite{MR2175409}).

\begin{thm}\label{T:AlmostEqualrk}
Suppose that
\[
    \dim V= \dim W+2
\]
is even. Let $\phi: \WW_{\BR}\rightarrow \mathrm{SO}(M)$ be a tempered L-parameter of $\Sp(W)$, where $M$ is the orthogonal rerpesentation of $\WW_\BR$ associated to $\phi$. Let $\pi\in \Pi_\phi(\Sp(W))$. If its small theta lift $\sigma\coloneqq \theta_{V,W,\psi}(\pi)$ is non-zero, then: 
\begin{enumerate}
    \item the restriction of $\sigma$ to $\SO(V)$ is irreducible, which we shall still denote by $\sigma$;

    \vskip 5pt

    \item as a representation of $\SO(V)$, we have $\sigma\in\Pi_{\theta(\phi)}(\SO(V))$, where $\theta(\phi)$ is a tempered L-parameter of $\SO(V)$ corresponding to the orthogonal representation
    \[
        \theta(M) = M \oplus \mathbbm{1}
    \]
    of $\WW_\BR$; moreover, if we let $c=\disc(V)$, then under the natural embedding $A_M^+\hookrightarrow A_{\theta(M)}^+$, we have
    \[
        \calJ_{\Fw_{-c}}\left(\sigma\right)\big|_{A^+_M}=\calJ_{\Fw}\left(\pi\right).
    \]
\end{enumerate}
\vskip 5pt
Moreover, for each $\sigma\in \Pi_{\theta(\phi)}(\SO(V))$, there exists a $\pi\in \Pi_{\phi}(\Sp(W))$ such that $\sigma\cong\theta_{V,W,\psi}(\pi)\big|_{\SO(V)}$.
\end{thm}

\vskip 5pt

\begin{rmk}
This theorem is essentially proved by A. Paul in \cite{MR2175409}. But we should remind readers that Paul's paper is not written in the fashion of the LLC as we stated. For the translation of Paul's result, readers can consult to \cite[Theorem 6.11]{chen2024arthur} and \cite[Theorem 8.1]{kakuhama2024local}. 
\end{rmk}

\vskip 5pt

\subsubsection{\bf Stable range theta lifts}

Besides the (almost) equal rank case, there is another extreme case, namely the stable range case. In this section, we shall recall some results on the stable range theta lifts.
Let $r_V$ be the Witt index of $V$, and similarly $r_W$ the Witt index of $W$. The following result of Loke-Ma \cite[Theorem A]{MR3305312} is a key ingredient in the proof of our main result.

\begin{thm}\label{Loke-Ma}
~
\begin{enumerate}
\item Suppose that $r_V\geq \dim W$.
Let $\pi$ be an irreducible unitary (genuine) representation of $\Mp(W)$ or $\Sp(W)$, depending on $\dim V$. Then its big theta lift $\Theta_{V,W,\psi}(\pi)$ is irreducible.

\vskip 5pt

\item Similarly, suppose that $r_W\geq \dim V$.
Let $\sigma$ be an irreducible unitary representation of $\Or(V)$. Then its big theta lift $\Theta_{W,V,\psi}(\sigma)$ is irreducible.
\end{enumerate}
\end{thm}

\vskip 5pt

Combining this irreducibility result with the induction principle, we can describe the stable range theta lifts under some technical conditions.

\begin{cor}\label{C:towers}
~
\begin{enumerate}
\item Fix an orthogonal space $V$ over $F$ with $\disc(V)=1$. Let $\pi$ be an irreducible unitary (genuine) representation of $\Mp(W)$ or $\Sp(W)$ depending on the parity of $\dim V$. Suppose that 
\[
    s=\frac{\dim V - \dim W -1}{2} \geq 0,
\]
and $\theta_{V,W,\psi}(\pi)$ is non-zero. For any $k\in\mathbb{N}$, let $V_k= V\oplus \mathcal{H}^k$, where $\mathcal{H}$ is the (orthogonal) hyperbolic plane. Then for any $k\geq \dim W$ there is an surjection
\[
    |\cdot|^{k+s-\frac{1}{2}}\times \cdots \times |\cdot|^{s+\frac{1}{2}} \rtimes \theta_{V,W,\psi}(\pi) \twoheadrightarrow \Theta_{V_k,W,\psi}(\pi).
\]

\vskip 5pt

\item Similarly, let $\sigma$ be an irreducible unitary representation of $\Or(V)$ with $\disc(V)=1$. Suppose that 
\[
    s'=\frac{\dim W - \dim V +1}{2} \geq 0,
\]
and $\theta_{W,V,\psi}(\sigma)$ is non-zero. For any $k\in\mathbb{N}$, let $W_k= W\oplus \mathcal{H}'^k$, where $\mathcal{H}'$ is the (symplectic) hyperbolic plane. Then for any $k\geq \dim V$ there is an surjection
\[
    |\cdot|^{k+s'-\frac{1}{2}}\times \cdots \times |\cdot|^{s'+\frac{1}{2}} \rtimes \theta_{W,V,\psi}(\sigma) \twoheadrightarrow \Theta_{W_k,V,\psi}(\sigma).
\]
\end{enumerate}
\end{cor}

\vskip 5pt

\begin{proof}
We only prove (1) here, the proof of (2) is similar. By the induction principle \cite[\S 5.2]{MR2175409}, we know that there is a non-zero map
\[
    \Theta_{V_k,W,\psi}(\pi) \longrightarrow \ind_{Q_k}^{\Or(V_k)}\left(|\det|^{-s-\frac{k}{2}}\boxtimes \theta_{V,W,\psi}(\pi)\right),
\]
where $Q_k$ is the standard parabolic subgroup of $\Or(V_k)$ with Levi component $\GL_{k}\times \Or(V)$. Since $k\geq \dim W$, the dual-pair $\Mp(W)\times \Or(V_k)$ (or $\Sp(W)\times \Or(V_k)$) is in the stable range. It follows from Theorem \ref{Loke-Ma} that $\Theta_{V_k,W,\psi}(\pi)$ is irreducible. Therefore the above map is injective. Applying both the MVW-involution functor (see \Cref{thm MVW}) and the contragredient functor to the above map, we deduce that $\Theta_{V_k,W,\psi}(\pi)$ is a quotient of $\Ind_{Q_k}^{\Or(V_k)}\left(|\det|^{s+\frac{k}{2}}\boxtimes \theta_{W,V,\psi}(\sigma)\right)$, which is also a quotient of 
\[
    |\cdot|^{k+s-\frac{1}{2}}\times \cdots \times |\cdot|^{s+\frac{1}{2}} \rtimes \theta_{V,W,\psi}(\pi) 
\]
as desired.

\end{proof}

\vskip 5pt

\subsection{A vanishing result for Hom spaces}

In this subsection we record a vanishing result for Hom spaces between certain parabolic induced representations and Langlands quotients of standard modules; this will be used in \Cref{section: Mackey} and \Cref{Descent along the Witt tower}.
 
\begin{pro}\label{vanishing lemma}
Let $V$ be either an orthogonal or symplectic space, $G=G(V)$ the isometry group (or the metaplectic/ Jacobi group when $V$ is symplectic), and 
\[
    \pi = LQ\left(\tau_1|\det|^{a_1}\times\cdots\times\tau_r|\det|^{a_r}\rtimes\pi_0\right)
\]
be an irreducible representation of $G$. Here 
\[
    a_1 \geq \cdots \geq a_r >0
\]
are real numbers, $\tau_i$'s are either unitary characters of $\GL_1(\BR)$ or discrete series of $\GL_2(\BR)$, and $\pi_0$ is a tempered representation of $G_0=G(V_0)$ for some non-degenerate subspace $V_0\subset V$ of the proper dimension. Then, for any:
\begin{itemize}
    \item $\tau$ unitary character of $\GL_1$ or discrete series of $\GL_2$;
    
    \vskip 5pt

    \item real number $a>a_1$;
    
    \vskip 5pt

    \item $\sigma$ a representation (not necessarily of finite length) of $G_1=G(V_1)$ for some non-degenerate subspace $V_1\subset V$ of the proper dimension,
\end{itemize}
\vskip 5pt
we have
\[
\Hom_{G}\left(\tau|\det|^a\rtimes\sigma,\pi\right) = 0.
\]
\end{pro}

\vskip 5pt

\begin{proof}
In the case that $G$ is a classical group or metaplectic group, this has been proved in \cite[Lemma 4.1.4]{chen2023multiplicity} and \cite[Theorem A.0.1]{chen2023multiplicity}. In the case that $G$ is a Jacobi group, this follows from the case of metaplectic group, combined with \Cref{Sun thm} and \Cref{F_k and G_k vs parabolic}.

\end{proof}

\vskip 10pt

\section{Uniqueness and epsilon dichotomy of the Fourier-Jacobi model}\label{section: stable 1}

We shall prove the uniqueness and epsilon dichotomy for the Fourier-Jacobi model in a Vogan L-packet in this section, and leave the proof of existence in the next section.

\subsection{Desideratum I: uniqueness and epsilon dichotomy}\label{Desideratum 1}
We retain the notation in \Cref{section FJ statment}. Our desideratum in this section is the following.

\begin{thm}\label{P:LtoR}
Let $\wt \phi: \WW_\BR\rightarrow \Sp(M)$ be a tempered L-parameter of $\Mp(W)$, and $\phi: \WW_\BR \rightarrow \SO(N)$ a tempered L-parameter of $\Sp(W)$, where $M$ is the symplectic representation of $\WW_\BR$ associated to $\wt\phi$, and $N$ is the orthogonal representation associated to $\phi$. If there exists some $\wt\pi\in\Pi_{\wt\phi}(\Mp(W))$ and $\pi\in\Pi_{\phi}(\Sp(W))$ such that 
\begin{equation}\label{assumption existense}
      \Hom_{\Mp(W)}\left(\wt \pi\boxtimes \pi, \omega_\psi\right)\neq 0,
\end{equation}
then we must have 
\[
\calJ_\psi(\wt \pi)\times \calJ_{\Fw}(\pi)=\eta_{N_1}\times \eta_M
\]
as in \Cref{GGP FJ}.
\end{thm}

\vskip 5pt

So now suppose that 
\[
\Hom_{\Mp(W)}\left(\wt \pi\boxtimes \pi, \omega_\psi\right)\cong 
\Hom_{\Mp(W)}\left(\left(\wt \pi\boxtimes \pi\right)\otimes \overline{\omega}_{\psi}, \BC\right) \neq 0.
\]
This assumption is equivalent to
\begin{equation*}
\Hom_{\Sp(W)}\left(\wt \pi \boxtimes \overline{\omega}_{\psi}, \pi^\vee\right) \neq 0.   
\end{equation*}
Here $\Sp(W)\hookrightarrow \Mp(W)\times_{\{\pm1\}} \Mp(W)$ is the diagonal embedding. By \eqref{Shimura}, there exists a unique $2n+1$-dimensional orthogonal space $V'$ with $\disc(V')=1$, such that $\wt \pi\cong \theta_{W,V',\psi}(\sigma')$ for an irreducible tempered representation $\sigma'$ of $\Or(V')$. Let $L$ be an $1$-dimensional quadratic space with $\disc(L)=-1$, $e\in L$ a generator of $L$, and 
\begin{equation}\label{VV'}
  V = V'\oplus L.  
\end{equation} 
We have $\disc(V)=1$ by \eqref{disc V odd} and \eqref{disc V even}. 
Consider the seesaw diagram:
\begin{equation*}\label{SeeSaw-0}\tag{$\natural.0$}
\begindc{\commdiag}[8]
\obj(-60,40)[a]{$\Mp(W)\times_{\{\pm1\}} \Mp(W)$}
\obj(-54,36)[aa]{$~$}
\obj(-60,-40)[b]{$\Sp(W)$}
\obj(-54,-36)[bb]{$~$}
\obj(60,40)[c]{$\Or(V)$}
\obj(54,36)[cc]{$~$}
\obj(60,-40)[d]{$\Or(V')\times \Or(L)$}
\obj(54,-36)[dd]{$~$}
\mor{a}{b}{}[\atleft,\solidline]
\mor{aa}{dd}{}[\atleft,\solidline]
\mor{c}{d}{}[\atleft,\solidline]
\mor{cc}{bb}{}[\atleft,\solidline]
\enddc
\qquad
\begindc{\commdiag}[8]
\obj(-60,40)[a]{$\Theta_{W,V'}(\sigma')\boxtimes {\omega}_{W,\psi_{-1}}$}
\obj(-54,36)[aa]{$~$}
\obj(-60,-40)[b]{$\pi^\vee$}
\obj(-54,-36)[bb]{$~$}
\obj(60,40)[c]{$\Theta_{V,W}(\pi^\vee)$}
\obj(54,36)[cc]{$~$}
\obj(60,-40)[d]{$\sigma'$}\obj(57,-38)[dd]{$~$}
\mor{a}{b}{}[\atleft,\solidline]
\mor{aa}{dd}{}[\atleft,\solidline]
\mor{c}{d}{}[\atleft,\solidline]
\mor{cc}{bb}{}[\atleft,\solidline]
\enddc
\end{equation*}
Then the associated seesaw identity reads:
\begin{equation*}\label{E:SS0}\tag{$\maltese.0$}
    \Hom_{\Sp(W)}\left(\Theta_{W,V',\psi}(\sigma')\boxtimes {\omega}_{\bar\psi}, \pi^\vee\right) = \Hom_{\Or(V')}\left(\Theta_{V,W,\psi}(\pi^\vee),\sigma'\right).
\end{equation*}
Note that $\omega_{\bar\psi}\simeq \overline{\omega}_{\psi}$, and $\wt\pi=\theta_{W,V',\psi}(\sigma')$ is a quotient of $\Theta_{W,V',\psi}(\sigma')$. We deduce from our assumption that 
\begin{equation}\label{E:SS01}
     \Hom_{\Or(V')}\left(\Theta_{V,W,\psi}(\pi^\vee),\sigma'\right) \neq 0,  
\end{equation}
and in particular $\sigma:=\theta_{V,W,\psi}(\pi^\vee)$ is non-zero. The key point of the proof of Theorem \ref{P:LtoR} is to show the following.

\begin{pro}\label{P:tempBessel}
In the context of above, we have
\[
\Hom_{\Or(V')}\left(\sigma,\sigma'\right)\neq 0.
\]
\end{pro}

\vskip 5pt

Note that if one can show that $\Theta_{V,W,\psi}(\pi^\vee)$ is irreducible, then this proposition would simply follow from \eqref{E:SS01}. When the base field is non-Archimedean, one can show the irreducibility of $\Theta_{V,W,\psi}(\pi^\vee)$ using Kudla's filtration (see \cite[Proposition C.4]{MR3166215}). However in the Archimedean case, the lack of Kudla's filtration makes it difficult to show the irreducibility of $\Theta_{V,W,\psi}(\pi^\vee)$ directly. 

\vskip 5pt

Assuming this proposition, we now end up this subsection with the proof of Theorem \ref{P:LtoR}. The proof of Proposition 
\ref{P:tempBessel} will be left to the rest subsections.

\begin{proof}[Proof of Theorem~\ref{P:LtoR}]
Let $\theta(\wt\phi)$ and $\theta(\phi)$ be the tempered L-parameters of $\SO(V')$ and $\SO(V)$ with associated orthogonal representations $M$ and $\theta(N) = N \oplus \mathbbm{1}$ of $\WW_{\BR}$, respectively.
By using the construction of the LLC for metaplectic groups (see \eqref{LLC for metaplectic} and \eqref{LLC for metaplectic character}), we know that the restriction of $\sigma'$ to $\SO(V')$ is irreducible and lies in the L-packet $\Pi_{\theta(\wt\phi)}(\SO(V'))$, and satisfies
\begin{equation}\label{distin charace SW}
    \calJ_{\mathfrak{w}'}\left(\sigma'\right)= \calJ_{\psi}\left(\wt\pi\right).
\end{equation}
Similarly, applying \Cref{T:AlmostEqualrk}, we know that the restriction of $\sigma$ to $\SO(V)$ is irreducible and lies in the L-packet $\Pi_{\theta(\phi)}(\SO(V))$, and satisfies
\begin{equation}\label{distin charace}
    \calJ_{\mathfrak w_{-1}}\left(\sigma\right)\big|_{A^+_N} = \calJ_{\mathfrak w}\left(\pi^\vee\right).
\end{equation}
Indeed, one can express $\calJ_{\mathfrak w}\left(\pi\right)$ in terms of $\calJ_{\mathfrak w_{1}}\left(\sigma\right)$:
\begin{equation}\label{distin charace 2}
\begin{aligned}
  \calJ_{\mathfrak w}\left(\pi\right)
    &= \calJ_{\mathfrak w}(\pi^\vee)\otimes \delta_{N,-1}
      &&\text{by \eqref{dual pi}} \\
    &= \calJ_{\mathfrak w_{-1}}\left(\sigma\right)\big|_{A^+_N}
       \otimes \delta_{N,-1}
      &&\text{by \eqref{distin charace}} \\
    &= \left(\calJ_{\mathfrak w_{1}}\left(\sigma\right)
       \otimes \delta_{N_1,-1}\right)\big|_{A^+_N}
       \otimes \delta_{N,-1}
      &&\text{by \eqref{change of Whittater data}} \\
    &= \calJ_{\mathfrak w_{1}}\left(\sigma\right)\big|_{A^+_N},
\end{aligned}
\end{equation}
where in the last step we have make use of the obvious equality $\delta_{N_1,-1}\big|_{A_N^+} = \delta_{N,-1}$.
On the other hand, by \Cref{P:tempBessel} we have
\[
  \Hom_{\SO(V')}\left(\sigma,\sigma'\right)\neq 0.
\]
Hence, by \Cref{GGP Bessel}, 
\begin{equation}\label{GGP Bessel epsilon}
\begin{cases}
  \calJ_{\mathfrak w_{1}}\left(\sigma\right)(a)
    = \epsilon\left(N_1^a\otimes M,\psi\right)
       \cdot \det\left(N_1^a\right)(-1)^{\frac{\dim M}{2}}
       \cdot \det\left(M\right)(-1)^{\frac{\dim N_1^a}{2}},
       & \mbox{for }a\in A_{N_1}^+;\\[5pt]
  \calJ_{\mathfrak w}\left(\sigma'\right)(b)
    = \epsilon\left(N_1\otimes M^b,\psi\right)
       \cdot \det\left(M^b\right)(-1)^{\frac{\dim N_1}{2}}
       \cdot \det\left(N_1\right)(-1)^{\frac{\dim M^b}{2}},
       & \mbox{for }b\in A_{M}.
\end{cases}
\end{equation}
Combining \eqref{distin charace SW}, \eqref{distin charace 2} and \eqref{GGP Bessel epsilon}, we obtain the desired equality in Theorem \ref{P:LtoR}. This completes the proof.

\end{proof}

\vskip 5pt

\subsection{Stable range seesaw I: some non-tempered Bessel models}
To prove Proposition \ref{P:tempBessel}, we shall make use of some stable range seesaw diagram. For any $k\in\mathbb{N}$, set
\[
    V_k^{\dagger}= V^{\dagger} + \mathcal{H}^k,
\]
where $\mathcal{H}$ is the (orthogonal) hyperbolic plane, and $\dagger \in\{\varnothing,\prime\}$. Note that both $\sigma$ and $\sigma'$ are tempered, we consider the Langlands quotient of the standard modules
\begin{equation}\label{sigmak}
       \sigma_k = LQ\left(|\cdot|^k\times\cdots\times|\cdot|^1\rtimes\sigma\right),
\end{equation}
and also
\begin{equation}\label{sigma'k}
        \sigma'_k = LQ\left(|\cdot|^{k-\frac{1}{2}}\times\cdots\times|\cdot|^{\frac{1}{2}}\rtimes\sigma'\right).
\end{equation}
According to Corollary \ref{C:towers}, we know that if $k\geq\dim W$, then $\sigma_k \cong \Theta_{V_k,W,\psi}\left(\pi^\vee\right)$ and $\sigma'_k\cong \Theta_{V'_k,W,\psi}\left(\wt \pi\right)$. Thus we fix a positive integer $t\geq\dim W$. Consider the following seesaw diagram:
\begin{equation*}\label{SeeSaw-1}\tag{$\natural.1$}
\begindc{\commdiag}[8]
\obj(-60,40)[a]{$\Mp(W)\times_{\{\pm1\}} \Mp(W)$}
\obj(-54,36)[aa]{$~$}
\obj(-60,-40)[b]{$\Sp(W)$}
\obj(-54,-36)[bb]{$~$}
\obj(60,40)[c]{$\Or(V_t)$}
\obj(54,36)[cc]{$~$}
\obj(60,-40)[d]{$\Or(V'_t)\times \Or(L)$}
\obj(54,-36)[dd]{$~$}
\mor{a}{b}{}[\atleft,\solidline]
\mor{aa}{dd}{}[\atleft,\solidline]
\mor{c}{d}{}[\atleft,\solidline]
\mor{cc}{bb}{}[\atleft,\solidline]
\enddc
\qquad 
\begindc{\commdiag}[8]
\obj(-60,40)[a]{$\Theta_{W,V'_t,\psi}(\sigma'_t)\boxtimes {\omega}_{\bar\psi}$}
\obj(-54,36)[aa]{$~$}
\obj(-60,-40)[b]{$\pi^\vee$}
\obj(-54,-36)[bb]{$~$}
\obj(60,40)[c]{$\Theta_{V_t,W,\psi}(\pi^\vee)$}
\obj(54,36)[cc]{$~$}
\obj(60,-40)[d]{$\sigma'_t$}
\obj(57,-38)[dd]{$~$}
\mor{a}{b}{}[\atleft,\solidline]
\mor{aa}{dd}{}[\atleft,\solidline]
\mor{c}{d}{}[\atleft,\solidline]
\mor{cc}{bb}{}[\atleft,\solidline]
\enddc
\end{equation*}
Then the associated seesaw identity reads:
\begin{equation}\label{E:SS9}\tag{$\maltese.1$}
    \Hom_{\Sp(W)}\left(\Theta_{W,V'_t,\psi}\left(\sigma'_t\right)\boxtimes {\omega}_{\bar\psi}, \pi^\vee\right) = \Hom_{\Or(V'_t)}\left(\Theta_{V_t,W,\psi}\left(\pi^\vee\right),\sigma'_t\right).
\end{equation}
Since $\wt\pi\cong\theta_{W,V'_t,\psi}\left(\sigma'_t\right)$ is a quotient of $\Theta_{W,V'_t,\psi}\left(\sigma'_t\right)$, and $\sigma_t\cong\Theta_{V_t,W,\psi}\left(\pi^\vee\right)$, the \Cref{E:SS9} and our assumption \eqref{assumption existense} imply that
\begin{equation}\label{nontempered Bessel}
     \Hom_{\Or(V'_t)}\left(\sigma_t,\sigma'_t\right) \neq 0.
\end{equation}
This Hom space is a non-tempered Bessel model. In the next section, we shall ``peel off'' the non-tempered part of \eqref{nontempered Bessel} to deduce that $\Hom_{\Or(V')}\left(\sigma,\sigma'\right) \neq 0$. 

\vskip 5pt

\subsection{Peeling off the non-tempered part I: Bessel model}\label{section: Mackey}

This section is devoted to proving the following lemma.

\begin{lem}\label{L:ind1}
In the context of the previous subsection, suppose that
\begin{equation}\label{assumption Bessel}
  \Hom_{\Or(V'_{k+1})}\left(\sigma_{k+1},\sigma'_{k+1}\right) \neq 0
\end{equation}
for some non-negative integer $k$. Then
\[
  \Hom_{\Or(V'_k)}\left(\sigma_k,\sigma'_k\right) \neq 0.
\]
\end{lem}

\begin{proof}
By \eqref{sigmak}, we have a natural surjection
\[
  |\cdot|^{k+1} \rtimes \sigma_k \twoheadrightarrow \sigma_{k+1}.
\]
Then assumption \eqref{assumption Bessel} implies that
\begin{equation}\label{total contribution bessel}
  \Hom_{\Or(V'_{k+1})}\left(|\cdot|^{k+1} \rtimes \sigma_k, \sigma'_{k+1}\right) \neq 0.
\end{equation}
We fix two vectors $e_{k+1},f_{k+1}\in V'_{k+1}$, such that they form a (orthogonal) hyperbolic plane, i.e., 
\[
    \left\langle e_{k+1},e_{k+1} \right\rangle = \left\langle f_{k+1},f_{k+1} \right\rangle = 0 \quad \mbox{and}\quad \left\langle e_{k+1},f_{k+1} \right\rangle = 1,
\]
and we have decompositions
\[
  V'_{k+1} = \left\langle e_{k+1} \right\rangle \oplus V'_k \oplus \left\langle f_{k+1} \right\rangle
  \quad \mbox{and} \quad
  V_{k+1} = \left\langle e_{k+1} \right\rangle \oplus V_k \oplus \left\langle f_{k+1} \right\rangle.
\]
Let $Q_{X}$ be the parabolic subgroup of $\Or(V_{k+1})$ stabilizing the isotropic line $X = \left\langle e_{k+1} \right\rangle$, with Levi component $M_{X} \cong \GL_1(\BR) \times \Or(V_k)$. Let $\mathcal X = \Or(V_{k+1})/Q_{X}$ be the flag variety of $\Or(V_{k+1})$
consisting of all isotropic lines in $V_{k+1}$, and $\mathcal E$ the $\Or(V_{k+1})$-equivariant tempered vector bundle on $\mathcal X$ whose fiber representation at $X\in\mathcal{X}$ is
\[
  \left(|\cdot|^{k+1} \boxtimes \sigma_k\right) \otimes \delta_{Q_{X}}^{\frac{1}{2}}.
\]
Here, the modulus character $\delta_{Q_{X}}$ factors through the projection to $\GL_1(\BR)$, and
\begin{equation}\label{modular orthogoal}
  \delta_{Q_{X}}(a) = |a|^{\dim V_k} \quad \mbox{for } a \in \GL_1(\mathbb R).
\end{equation}
Then we have an isomorphism of $\Or(V_{k+1})$-representations
\begin{equation}\label{bessel iso}
   \Gamma^{\mathcal S}(\mathcal X, \mathcal E) \cong |\cdot|^{k+1} \rtimes \sigma_k.   
\end{equation}

\vskip 5pt

Now we can appeal to the Mackey theory (or more precisely, Borel's lemma \Cref{P:BorelLemma}). There are two $\Or(V'_{k+1})$-orbits on $\mathcal X$:
\begin{itemize}
  \item the closed orbit $\mathcal Z$, consisting of isotropic lines in $V_{k+1}$
    that are contained in $V'_{k+1}$;
  \vskip 5pt
  \item the open orbit $\mathcal U$, consisting of isotropic lines in $V_{k+1}$ that
    are not contained in $V'_{k+1}$.
\end{itemize}
We have a short exact sequence of $\Or(V'_{k+1})$-representations
\begin{equation*}
  0 \longrightarrow
  \Gamma^{\mathcal S}(\mathcal U, \mathcal E)
  \longrightarrow
  \Gamma^{\mathcal S}(\mathcal X,\mathcal E)
  \longrightarrow
  \Gamma^{\mathcal S}_{\mathcal Z}(\mathcal X,\mathcal E)
  \longrightarrow 0,
\end{equation*}
where $\Gamma^{\mathcal S}_{\mathcal Z}(\mathcal X,\mathcal E)\coloneqq \Gamma^{\mathcal S}(\mathcal X,\mathcal E)/\Gamma^{\mathcal S}(\mathcal U, \mathcal E)$. This induces an exact sequence of Hom spaces
\begin{equation*}
  \Hom_{\Or(V'_{k+1})}\left(\Gamma^{\mathcal S}_{\mathcal Z}\left(\mathcal X,\mathcal E\right), \sigma'_{k+1}\right)
  \longrightarrow
  \Hom_{\Or(V'_{k+1})}\left(\Gamma^{\mathcal S}\left(\mathcal X,\mathcal E\right),\sigma'_{k+1}\right)
  \longrightarrow
  \Hom_{\Or(V'_{k+1})}\left(\Gamma^{\mathcal S}\left(\mathcal U, \mathcal E\right), \sigma'_{k+1}\right).
\end{equation*}
We claim that the open orbit has non-zero contribution, namely
\[
    \Hom_{\Or(V'_{k+1})}\left(\Gamma^{\mathcal S}\left(\mathcal U, \mathcal E\right), \sigma'_{k+1}\right) \neq 0.
\]
Note that by \eqref{total contribution bessel}, \eqref{bessel iso} and the exact sequence above, to show our claim, it suffices to show that 
\[
    \Hom_{\Or(V'_{k+1})}\left(\Gamma^{\mathcal S}_{\mathcal Z}\left(\mathcal X,\mathcal E\right), \sigma'_{k+1}\right)=0.
\]
We analyze $\Gamma^{\mathcal S}_{\mathcal Z}\left(\mathcal X,\mathcal E\right)$ using \Cref{P:BorelLemma}. Choose the representative $X = \left\langle e_{k+1}\right\rangle \in \mathcal Z$. Then its stabilizer $\Stab_{\Or(V'_{k+1})}(X) = Q'_{X}$ is a parabolic subgroup of $\Or(V'_{k+1})$, with Levi component $M'_{X} \cong \GL_1(\BR) \times \Or(V'_k)$. Hence $\mathcal{E}\big|_{\mathcal{Z}}$ is isomorphic to a $\Or(V'_{k+1})$-equivariant bundle over $\mathcal{Z} \cong \Or(V'_{k+1})/ Q'_{X}$ whose fiber representation at $X\in\mathcal{Z}$ is
\[
  \left(\left(|\cdot|^{k+1} \boxtimes \sigma_k\right) \otimes \delta_{Q_{X}}^{\frac{1}{2}}\right)\Big|_{Q'_{X}}
  \cong |\cdot|^{k+1+\frac{\dim V_k}{2}}\boxtimes \left(\sigma_k\big|_{\Or(V'_k)}\right)
  \qquad \text{(by \eqref{modular orthogoal})}.
\]
Note that the modulus character $\delta_{Q'_{X}}$ factors through the projection to $\GL_1(\BR)$, and we have $\delta_{Q'_{X}}(a) = |a|^{\dim V'_k}$ for $a \in \GL_1(\mathbb R)$.
It follows that
\begin{align*}
  \Gamma^{\mathcal S}\left(\mathcal{Z},\mathcal{E}\big|_{\mathcal{Z}}\right)
  &\cong 
  \unind_{Q'_{X}}^{\Or(V'_{k+1})}\left(|\cdot|^{k+1+\frac{\dim V_k}{2}}\boxtimes \left(\sigma_k\big|_{\Or(V'_k)}\right)\right) \\
  &\cong \ind_{Q'_{X}}^{\Or(V'_{k+1})}\left(|\cdot|^{k+\frac{3}{2}}\boxtimes \left(\sigma_k\big|_{\Or(V'_k)}\right)\right) \\ 
  &= |\cdot|^{k+\frac{3}{2}} \rtimes \left(\sigma_k\big|_{\Or(V'_k)}\right)
\end{align*}
as $\Or(V'_{k+1})$-representations. On the other hand, the complexified conormal bundle $\mathcal N_{\mathcal Z|\mathcal X}^*$ is a $\Or(V'_{k+1})$-equivariant bundle over $\mathcal Z$ whose fiber representation at $X\in\mathcal{Z}$ is $\sgn|\cdot| \boxtimes\mathbbm{1}_{\Or(V'_k)}$. By \Cref{P:BorelLemma}, there is an $\Or(V'_{k+1})$-equivariant filtration on $\Gamma^{\mathcal S}_{\mathcal Z}(\mathcal X,\mathcal E)$, with graded pieces
\[
  \Gamma^{\mathcal S}\left(\mathcal Z, \Sym^j\mathcal N_{\mathcal Z|\mathcal X}^* \otimes \mathcal E\big|_\mathcal Z\right) \cong
  \left(\sgn^j |\cdot|^{k+\frac{3}{2}+j}\right)\rtimes\left(\sigma_k\big|_{\Or(V'_k)}\right)
  \quad \mbox{for } j \in \mathbb N.
\]
Recall that in the construction of $\sigma'_{k+1}$ (see \eqref{sigma'k}), all ``exponents'' in the standard module of $\sigma'_{k+1}$ are bounded above by $k+\frac{1}{2}$, hence we are in a situation that \Cref{vanishing lemma} is applicable. It follows that
\begin{equation*}
  \Hom_{\Or(V'_{k+1})}\left(\left(\sgn^j |\cdot|^{k+\frac{3}{2}+j}\right)\rtimes\left(\sigma_k\big|_{\Or(V'_k)}\right), \sigma'_{k+1}\right)
  = 0
\end{equation*}
for all $j \in \mathbb N$. We conclude that $\Hom_{\Or(V'_{k+1})}\left(\Gamma^{\mathcal S}_{\mathcal Z}\left(\mathcal X,\mathcal E\right), \sigma'_{k+1}\right)=0$ and our claim holds.

\vskip 5pt

Next we compute the contribution of the open orbit. Recall that $V = V' \oplus \left\langle e\right\rangle$ with $\left\langle e,e\right\rangle_V = -1$ (see \eqref{VV'}). Choose a representative
\[
  Y = \left\langle e + \frac{e_{k+1}+f_{k+1}}{\sqrt{2}} \right\rangle \in \mathcal U.
\]
Then $\Stab_{\Or(V'_{k+1})}(Y) \cong \Or(V_k)$, and $\mathcal{E}\big|_{\mathcal U}$ is isomorphic to a $\Or(V'_{k+1})$-equivariant bundle over $\mathcal U$ whose fiber representation at $Y$ is $\sigma_k$. Therefore,
\[
  \Gamma^{\mathcal S}(\mathcal U, \mathcal E) \cong \unind_{\Or(V_k)}^{\Or(V'_{k+1})}\left(\sigma_k\right).
\]
By the Frobenius reciprocity in \Cref{Frobenious cor}, we obtain that
\begin{equation*}
\begin{split}
  \Hom_{\Or(V'_{k+1})}\left(\Gamma^{\mathcal S}(\mathcal U, \mathcal E), \sigma'_{k+1}\right)
  &\cong
  \Hom_{\Or(V'_{k+1})}\left(\unind_{\Or(V_k)}^{\Or(V'_{k+1})}\left(\sigma_k\right), \sigma'_{k+1}\right)
  \\
  & \cong  \Hom_{\Or(V_k)}\left((\sigma'_{k+1})^\vee, \sigma^\vee_k\right)\\
  &\cong
  \Hom_{\Or(V_k)}\left(\sigma'_{k+1}, \sigma_k\right).
\end{split}
\end{equation*}
Here in the last isomorphism, we also make use of the fact that irreducible representations of orthogonal groups are self-dual. Our claim in the previous paragraph then implies that 
\[
  \Hom_{\Or(V_k)}\left(\sigma'_{k+1}, \sigma_k\right) \neq 0.
\]
Note that by our construction \eqref{sigma'k}, there is also a natural surjection
\[
    |\cdot|^{k+\frac{1}{2}} \rtimes \sigma'_k \twoheadrightarrow \sigma'_{k+1}.
\]
Applying the same argument above using this surjection, we obtain that
\[
  \Hom_{\Or(V'_k)}\left(\sigma_{k}, \sigma'_k\right) \neq 0
\]
as desired.

\end{proof}

\vskip 5pt

Using \Cref{L:ind1} iteratively, we conclude from \eqref{nontempered Bessel} that $\Hom_{\Or(V')}\left(\sigma,\sigma'\right) \neq 0$. This completes the proof of \Cref{P:tempBessel}.

\section{Existence of the Fourier-Jacobi model}\label{section: stable 2}
We shall prove the existence of the Fourier-Jacobi model in a Vogan L-packet in this section. 

\subsection{Desideratum II: Existence}
We retain the notations in \Cref{Desideratum 1}. Our desideratum in this section is the following.
\begin{thm}\label{P:RtoL}
Let $\wt \phi: \WW_\BR\rightarrow \Sp(M)$ be a tempered L-parameter of $\Mp(W)$, and $\phi: \WW_\BR \rightarrow \SO(N)$ a tempered L-parameter of $\Sp(W)$, where $M$ is the symplectic representation of $\WW_\BR$ associated to $\wt\phi$, and $N$ is the orthogonal representation associated to $\phi$. Then there exists an $\wt\pi\in\Pi_{\wt\phi}(\Mp(W))$ and $\pi\in\Pi_{\phi}(\Sp(W))$, such that the Hom space
\begin{equation*}
      \Hom_{\Mp(W)}\left(\wt \pi\boxtimes \pi, \omega_\psi\right)\neq 0.
\end{equation*}
\end{thm}

\vskip 5pt

Similar to the proof of \Cref{P:LtoR}, we first write 
\[
\Hom_{\Mp(W)}\left(\wt \pi\boxtimes \pi, \omega_\psi\right)\cong 
\Hom_{\Sp(W)}\left(\wt \pi \boxtimes \overline{\omega}_{\psi}, \pi^\vee\right).
\]
We will use Bessel models and a seesaw identity argument to produce a non-zero element in the right hand side Hom space of the above equation. 

\vskip 5pt

We fix a $2n+1$-dimensional quadratic space $V'$ of discriminant $+1$, and set $V\coloneqq V'\oplus L$, where $L$ is an $1$-dimensional quadratic space with $\disc(L)=-1$. Let $\theta(\wt\phi)$ and $\theta(\phi)$ be the tempered L-parameters of $\SO(V')$ and $\SO(V)$ with associated orthogonal representations $M$ and $\theta(N) = N \oplus \mathbbm{1}$ of $\WW_{\BR}$, respectively. Then by \Cref{GGP Bessel}, there exists a relevant pure inner form of $\SO(V')\times \SO(V)$, corresponding to a pair of quadratic spaces $V'_\alpha\subset V_\alpha$, together with irreducible representations $\sigma'_\alpha\in \Pi_{\theta(\wt \phi)}(\SO(V'_\alpha))$ and $\sigma_\alpha \in \Pi_{\theta(\phi)}(\SO(V_\alpha))$, such that  
\[
\Hom_{\SO(V'_\alpha)}\left(\sigma_\alpha , \sigma'_\alpha\right)\neq 0.
\]
By \Cref{T:AlmostEqualrk}, we know that there exists an irreducible representation $\pi$ of $\Sp(W)$, such that $\pi^\vee\in \Pi_\phi(\Sp(W))$, and 
\[
\sigma_\alpha =\theta_{V,W,\psi}\left(\pi^\vee\right)\big|_{\SO(V_\alpha)}.
\]
It follows from \cite[Theorem 4.9]{MR3194648}) that one also has $\pi\in \Pi_{\phi}(\Sp(W))$. Let $\sigma=\theta_{V,W,\psi}(\pi^\vee)$ be an irreducible representation of $\Or(V_\alpha)$. Then we have  
\[
\Hom_{\SO(V'_\alpha)}\left(\sigma_\alpha,\sigma'_\alpha\right)\cong  \Hom_{\Or(V'_\alpha)}\left(\sigma,\Ind_{\SO(V'_\alpha)}^{\Or(V'_\alpha)}\left(\sigma'_\alpha\right)\right).
\]
This implies that there exists an irreducible summand $\sigma'$ of $\Ind_{\SO(V'_\alpha)}^{\Or(V'_\alpha)}\left(\sigma'_\alpha\right)$ such that 
\begin{equation}\label{exten to Or}
 \Hom_{\Or(V'_\alpha)}\left(\sigma,\sigma'\right)\neq 0.   
\end{equation}
Now consider the seesaw diagram as in \eqref{SeeSaw-0}. The associated seesaw identity reads
\begin{equation*}
    \Hom_{\Sp(W)}\left(\Theta_{W,V'_\alpha,\psi}(\sigma')\boxtimes {\omega}_{\bar\psi}, \pi^\vee\right) = \Hom_{\Or(V'_\alpha)}\left(\Theta_{V_\alpha,W,\psi}\left(\pi^\vee\right),\sigma'\right).
\end{equation*}
Note that ${\omega}_{\bar\psi}\cong \overline{\omega}_{\psi}$, and $\sigma$ is a quotient of $\Theta_{V_\alpha,W,\psi}\left(\pi^\vee\right)$. 
We conclude from \eqref{exten to Or} and above seesaw identity that 
\begin{equation*}
     \Hom_{\Sp(W)}\left(\Theta_{W,V'_\alpha,\psi}(\sigma')\boxtimes {\omega}_{\bar\psi}, \pi^\vee\right)\neq 0.
\end{equation*}
In particular, we know that $\Theta_{W,V'_\alpha,\psi}(\sigma')\neq 0$. Let $\wt \pi=\theta_{W,V'_\alpha,\psi}(\sigma')$. It follows from \eqref{LLC for metaplectic} that $\wt \pi \in \Pi_{\wt \phi}(\Mp(W))$. Again, due to the issue of big theta lift, we are not able to conclude 
\[
  \Hom_{\Sp(W)}\left(\wt \pi\boxtimes\overline{\omega}_{\psi}, \pi^\vee\right) \neq 0
\]
from the non-vanishing of the Hom space on the left hand side of above seesaw identity.

\vskip 5pt

\subsection{Stable range seesaw II: some non-tempered Fourier-Jacobi models}

As in \Cref{section: stable 1}, we circumvent the issue of the big theta lift by working with some stable range theta lifts.
For any non-negative integer $k$, we put 
\[
    W_k = W \oplus  \mathcal H'^k,
\]
where $\mathcal H'$ is the (symplectic) hyperbolic plane. Note that both $\wt \pi$ and $\pi$ are tempered, we consider the Langlands quotient of the standard modules
\begin{equation}\label{Theta tower O to Mp}
     \wt \pi_k = LQ\left(|\cdot|^{k-\frac{1}{2}}\times\cdots\times |\cdot|^{\frac{1}{2}}\rtimes \wt \pi\right), 
\end{equation}
and also
\begin{equation}\label{Theta tower O to Sp}
     \pi^\vee_k = LQ\left(|\cdot|^{k-1}\times\cdots\times |\cdot|^{1}\rtimes\pi^\vee_1 \right),
\end{equation}
where $\pi_1^\vee\coloneqq\theta_{W_1,V_\alpha,\psi}(\sigma)$ is tempered. Since $\sigma$ has non-zero theta lift to $\Sp(W)$, by \cite[Proposition C.1]{MR3166215} one knows that $\pi_1^\vee$ is a direct summand of $|\cdot|^0\rtimes\pi^\vee$. According to Corollary \ref{C:towers}, we know that if $k\geq\dim V$, then $\pi^\vee_k \cong \Theta_{W_k,V_\alpha,\psi}(\sigma)$ and $\wt \pi_k \cong \Theta_{W_k,V'_\alpha,\psi}(\sigma')$. We fix a positive integer $t>\dim V$. Consider the following seesaw diagram:

\begin{equation*}\label{SeeSaw-2}\tag{$\natural.2$}
\begindc{\commdiag}[8]
\obj(-60,40)[a]{$\Mp(W_t)\times_{\{\pm1\}} \Mp(W_t)$}
\obj(-54,36)[aa]{$~$}
\obj(-60,-40)[b]{$\Sp(W_t)$}
\obj(-54,-36)[bb]{$~$}
\obj(60,40)[c]{$\Or(V_\alpha)$}
\obj(54,36)[cc]{$~$}
\obj(60,-40)[d]{$\Or(V'_\alpha)\times \Or(L)$}
\obj(54,-36)[dd]{$~$}
\mor{a}{b}{}[\atleft,\solidline]
\mor{aa}{dd}{}[\atleft,\solidline]
\mor{c}{d}{}[\atleft,\solidline]
\mor{cc}{bb}{}[\atleft,\solidline]
\enddc
\qquad 
\begindc{\commdiag}[8]
\obj(-60,40)[a]{$\Theta_{W_t,V'_\alpha,\psi}(\sigma')\boxtimes {\omega}_{\bar\psi}$}
\obj(-54,36)[aa]{$~$}
\obj(-60,-40)[b]{$\pi^\vee_t$}
\obj(-54,-36)[bb]{$~$}
\obj(60,40)[c]{$\Theta_{V_\alpha,W_t,\psi}(\pi^\vee_t)$}
\obj(54,36)[cc]{$~$}
\obj(60,-40)[d]{$\sigma'$}\obj(57,-38)[dd]{$~$}
\mor{a}{b}{}[\atleft,\solidline]
\mor{aa}{dd}{}[\atleft,\solidline]
\mor{c}{d}{}[\atleft,\solidline]
\mor{cc}{bb}{}[\atleft,\solidline]
\enddc
\end{equation*}
Then the associated seesaw identity reads:
\begin{equation*}\label{E:SS2}\tag{$\maltese.2$}
    \Hom_{\Sp(W_t)}\left(\Theta_{W_t,V'_\alpha,\psi}(\sigma')\boxtimes \overline{\omega}_{\psi}, \pi^\vee_t \right) = \Hom_{\Or (V'_\alpha)}\left(\Theta_{V_\alpha,W_t,\psi}(\pi^\vee_t),\sigma'\right).
\end{equation*}
Note that $\Theta_{W_t,V'_\alpha,\psi}(\sigma')\cong \wt\pi_t$, and $\sigma$ is a quotient of $\Theta_{V_\alpha,W_t,\psi}(\pi^\vee_t)$. Therefore, we deduce from \eqref{exten to Or} and \eqref{E:SS2} that 
\[
    \Hom_{\Sp(W_t)}\left(\wt\pi_t\boxtimes \overline{\omega}_{\psi}, \pi^\vee_t \right) \neq 0.
\]
In the next subsection, we shall use an analog of Lemma \ref{L:ind1} to ``peel off'' the non-tempered part of the Hom space above.

\vskip 5pt

\subsection{Peeling off the non-tempered part II: Fourier--Jacobi model}\label{Descent along the Witt tower}

This section is devoted to proving the following lemma. 
\begin{lem}\label{L:ind2}
Suppose that
\begin{equation}\label{assuption FJ}
  \Hom_{\Sp(W_{k+1})}\left(\wt \pi_{k+1}\boxtimes \overline{\omega}_{\psi}, \pi_{k+1}^{\vee}\right) \neq 0
\end{equation}
for some non-negative integer $k$. Then
\[
  \Hom_{\Sp(W_{k})}\left(\wt \pi_{k}\boxtimes \overline{\omega}_{\psi}, \pi_{k}^{\vee}\right) \neq 0.
\]
\end{lem}

\vskip 5pt

\begin{proof}
The proof of this lemma is parallel to that of \Cref{L:ind1}; we therefore only sketch it. By \Cref{Sun thm} we know that when restricted to $\Sp(W_{k+1})$, one has
\[
  \wt\pi_{k+1}\boxtimes \overline{\omega}_{\psi} \cong \mathcal{F}_{\bar\psi}\left(\wt\pi_{k+1}\right).
\]
As explained in \eqref{fj MODEL rewrite}, we can rephrase \eqref{assuption FJ} as a branching problem from the Jacobi group $\J(W_{k+1})$ to the symplectic group $\Sp(W_{k+1})$ through the natural embedding $\Sp(W_{k+1})\hookrightarrow \J(W_{k+1})$. By \Cref{F_k and G_k vs parabolic} and the construction of $\wt\pi_{k+1}$ (see \eqref{Theta tower O to Mp}), there is a
surjection
\[
  |\cdot|^{k+\frac{1}{2}}\rtimes \mathcal{F}_{\bar\psi}\left(\wt\pi_{k}\right) \twoheadrightarrow \mathcal{F}_{\bar\psi}\left(\wt\pi_{k+1}\right),
\]
where we recall that $|\cdot|^{k+\frac{1}{2}}\rtimes \mathcal{F}_{\bar\psi}\left(\wt\pi_{k}\right)$ on the left hand side is the normalized ``parabolic induction'' of $\J(W_{k+1})$ (see \eqref{normalized Schwartz parabolic induction Jacobi} and \eqref{modular character Q_J}). Our assumption \eqref{assuption FJ} then implies that
\begin{equation}\label{k+1 piece}
  \Hom_{\Sp(W_{k+1})}\left(|\cdot|^{k+\frac{1}{2}}\rtimes \mathcal{F}_{\bar\psi}\left(\wt\pi_{k}\right), \pi_{k+1}^{\vee} \right) \neq 0.
\end{equation}
We fix two vectors $e_{k+1},f_{k+1}\in W_{k+1}$, such that they form a (symplectic) hyperbolic plane, and we have a decomposition
\[
  W_{k+1} = \left\langle e_{k+1}\right\rangle \oplus W_k \oplus \left\langle f_{k+1}\right\rangle.
\]
Let $\mathcal X = \mathcal F^{\J}_{W_{k+1},1}$ be the ``flag variety'' of $\J(W_{k+1})$ defined in \eqref{flag J}; it consists of pairs $\left(X,v^*\right)$ where $X$ is an isotropic line in $W_{k+1}$ and $v^*\in X^*$.
We fix a base point $\left(Y, 0\right)\in \mathcal F^{\J}_{W_{k+1},1}$, where $Y = \left\langle e_{k+1}\right\rangle$, and let $P_{Y}=\Stab_{\J(W_{k+1})}\left(Y,0\right)$ be a ``parabolic subgroup''  of $\J(W_{k+1})$, with ``Levi quotient'' $M_{Y} \cong \GL_1(\BR) \times \J(W_k)$. Let $\mathcal E$ be the $\J(W_{k+1})$-equivariant tempered vector bundle on $\mathcal X$ whose fiber representation at $\left(Y,0\right)\in\mathcal{X}$ is 
\[
  \left(|\cdot|^{k+\frac{1}{2}}\boxtimes \mathcal{F}_{\bar\psi}\left(\wt\pi_{k}\right) \right)
  \otimes \delta_{P_{Y}}^{\frac{1}{2}}.
\]
Then we have
\[
  \Gamma^\mathcal{S}(\mathcal X, \mathcal E) \cong |\cdot|^{k+\frac{1}{2}}\rtimes \mathcal{F}_{\bar\psi}\left(\wt\pi_{k}\right)
\]
as $\J(W_{k+1})$-representations. There are two $\Sp(W_{k+1})$-orbits on $\mathcal X$:
\begin{itemize}
  \item the closed orbit $\mathcal Z$, consisting of pairs $\left(X,0\right)$;
  
  \vskip 5pt
  
  \item the open orbit $\mathcal U$, consisting of pairs $\left(X,v^*\right)$ with $v^*\neq 0$.
\end{itemize}
\vskip 5pt
This yields an $\Sp(W_{k+1})$-equivariant short exact sequence
\begin{equation*}
  0 \longrightarrow
  \Gamma^{\mathcal S}(\mathcal U, \mathcal E)
  \longrightarrow
  \Gamma^{\mathcal S}(\mathcal X,\mathcal E)
  \longrightarrow
  \Gamma^{\mathcal S}_{\mathcal Z}(\mathcal X,\mathcal E)
  \longrightarrow 0,
\end{equation*}
where $\Gamma^{\mathcal S}_{\mathcal Z}(\mathcal X,\mathcal E)\coloneqq \Gamma^{\mathcal S}(\mathcal X,\mathcal E)/\Gamma^{\mathcal S}(\mathcal U, \mathcal E)$. It induces an exact sequence of Hom spaces
\begin{equation*}
  \Hom_{\Sp(W_{k+1})}\left(\Gamma^{\mathcal S}_{\mathcal Z}\left(\mathcal X,\mathcal E\right), \pi_{k+1}^{\vee}\right)
  \longrightarrow
  \Hom_{\Sp(W_{k+1})}\left(\Gamma^{\mathcal S}\left(\mathcal X,\mathcal E\right),\pi_{k+1}^{\vee}\right)
  \longrightarrow
  \Hom_{\Sp(W_{k+1})}\left(\Gamma^{\mathcal S}\left(\mathcal U, \mathcal E\right), \pi_{k+1}^{\vee}\right).
\end{equation*}
We claim that the open orbit has non-zero contribution, namely
\[
    \Hom_{\Sp(W_{k+1})}\left(\Gamma^{\mathcal S}\left(\mathcal U, \mathcal E\right), \pi_{k+1}^{\vee}\right) \neq 0.
\]
Note that by \eqref{k+1 piece} and the exact sequence above, it suffices to show that 
\[
    \Hom_{\Sp(W_{k+1})}\left(\Gamma^{\mathcal S}_{\mathcal Z}\left(\mathcal X,\mathcal E\right), \pi_{k+1}^{\vee}\right)=0.
\]

\vskip 5pt

We analyze $\Gamma^{\mathcal S}_{\mathcal Z}\left(\mathcal X,\mathcal E\right)$ using \Cref{P:BorelLemma}. Choose the representative $\left(Y,0\right)\in \mathcal Z$. Then
\[
    Q_{Y}\coloneqq\Stab_{\Sp(W_{k+1})}\left(Y,0\right)
\]
is the parabolic subgroup of $\Sp(W_{k+1})$ stabilizing $Y$, with Levi component $\GL_1(\BR)\times \Sp(W_k)$. The modulus character of $Q_{Y}$ factor through $\GL_1(\BR)$, and $\delta_{Q_{Y}}(a) = |a|^{\dim W_k+2}$ for $a\in \GL_1(\mathbb R)$. It follows that
\begin{align*}
  \Gamma^{\mathcal S}\left(\mathcal{Z},\mathcal{E}\big|_{\mathcal{Z}}\right)
  &\cong 
  \unind_{Q_Y}^{\Sp(W_{k+1})}\left(\left(|\cdot|^{k+\frac{1}{2}}\boxtimes\mathcal{F}_{\bar\psi}\left(\wt\pi_k\right)\right) \otimes \delta_{P_{Y}}^{\frac{1}{2}}\big|_{Q_{Y}}\right) \\
  &\cong \unind_{Q_Y}^{\Sp(W_{k+1})}\left(|\cdot|^{k+2+\frac{\dim W_k}{2}}\boxtimes\left(\mathcal{F}_{\bar\psi}\left(\wt\pi_k\right)\big|_{\Sp(W_k)}\right)\right) \\
  &\cong |\cdot|^{k+1}\rtimes \left(\mathcal{F}_{\bar\psi}\left(\wt\pi_k\right)\big|_{\Sp(W_k)}\right)
\end{align*}
as $\Sp(W_{k+1})$-representations. On the other hand, the complexified conormal bundle $\mathcal N_{\mathcal{Z}|\mathcal{X}}^*$ is a $\Sp(W_{k+1})$-equivariant bundle over $\mathcal Z$ whose fiber representation at $\left(Y,0\right)\in\mathcal{Z}$ is $\sgn|\cdot|\boxtimes \mathbbm{1}_{\Sp(W_k)}$. By \Cref{P:BorelLemma}, there is an $\Sp(W_{k+1})$-equivariant filtration on $\Gamma^{\mathcal S}_{\mathcal Z}(\mathcal X,\mathcal E)$, with graded pieces
\[
  \Gamma^{\mathcal S}\left(\mathcal Z, \Sym^j\mathcal N_{\mathcal Z|\mathcal X}^* \otimes \mathcal E\big|_\mathcal Z\right) \cong
  \left(\sgn^j|\cdot|^{k+j+1}\right)\rtimes \left(\mathcal{F}_{\bar\psi}\left(\wt\pi_k\right)\big|_{\Sp(W_k)}\right).
\]
Recall that in the construction of $\pi^\vee_{k+1}$ (see \eqref{Theta tower O to Sp}), all ``exponents'' in the standard module of $\pi^\vee_{k+1}$ are bounded above by $k$, hence we are in a situation that \Cref{vanishing lemma} is applicable. It follows that
\begin{equation*}
  \Hom_{\Sp(W_{k+1})}\left(\left(\sgn^j|\cdot|^{k+j+1}\right)\rtimes \left(\mathcal{F}_{\bar\psi}\left(\wt\pi_k\right)\big|_{\Sp(W_k)}\right), \pi_{k+1}^{\vee}\right) = 0
\end{equation*}
for all $j \in \mathbb N$. We conclude that $\Hom_{\Sp(W_{k+1})}\left(\Gamma^{\mathcal S}_{\mathcal Z}\left(\mathcal X,\mathcal E\right), \pi_{k+1}^{\vee}\right)=0$ and our claim holds.

\vskip 5pt

Now, to analyze the contribution of the open orbit, we choose the representative $\left(Y, f_{k+1}\right)\in \mathcal{U}$, where $f_{k+1}$ is regarded as a linear functional on $Y$ through the pairing on $W_{k+1}$. By \eqref{embedding J to Sp} we have
\[
  \Stab_{\Sp(W_{k+1})}\left(Y, f_{k+1}\right)\cong \J\left(W_k\right).
\]
Hence $\mathcal E\big|_{\mathcal U}$ is isomorphic to a $\Sp(W_{k+1})$-equivariant bundle over $\mathcal U$ whose fiber representation at $\left(Y, f_{k+1}\right)$ is
\[
  \left(|\cdot|^{k+\frac{1}{2}}\boxtimes\mathcal{F}_{\bar\psi}\left(\wt \pi_{k}\right)\right)\otimes \delta_{P_{Y}}^{\frac{1}{2}}\big|_{\J(W_k)}
  \cong \mathcal{F}_{\bar\psi}\left(\wt \pi_{k}\right).
\]
Therefore
\[
  \Gamma^\mathcal{S}(\mathcal U, \mathcal E) \cong \unind_{\J(W_{k})}^{\Sp(W_{k+1})}\mathcal{F}_{\bar\psi}\left(\wt \pi_{k}\right).
\]
By \Cref{Frobenious}, one has
\begin{equation*}
\begin{split}
  \Hom_{\Sp(W_{k+1})}\left(\Gamma^\mathcal{S}(\mathcal U, \mathcal E), \pi_{k+1}^{\vee}\right)
  &\cong \Hom_{\Sp(W_{k+1})}\left(\unind_{\J(W_{k})}^{\Sp(W_{k+1})}\mathcal{F}_{\bar\psi}\left(\wt \pi_{k}\right),\pi_{k+1}^{\vee}\right)\\
  &\cong \Hom_{\Sp(W_{k+1})}\left(\pi_{k+1}\otimes\unind_{\J(W_{k})}^{\Sp(W_{k+1})}\mathcal{F}_{\bar\psi}\left(\wt \pi_{k}\right), \mathbbm{1}_{\Sp(W_{k+1})}\right)\\ 
  &\cong \Hom_{\J(W_{k})}\left(\pi_{k+1}\big|_{\J(W_k)}\otimes\mathcal{F}_{\bar\psi}\left(\wt \pi_{k}\right), \mathbbm{1}_{\J(W_{k})}\right).
\end{split}
\end{equation*}
Note that the last Hom space can be rewritten as
\[
\Hom_{\J(W_{k})}\left(\pi_{k+1}\big|_{\J(W_k)}\otimes\mathcal{F}_{\bar\psi}\left(\wt \pi_{k}\right), \mathbbm{1}_{\J(W_{k})}\right)
  \cong \Hom_{\J(W_{k})}\left(\left(\pi_{k+1}\big|_{\J(W_k)}\right)_{Z,\psi}\otimes\left(\overline{\omega}_\psi\otimes\wt\pi_{k}\right), \mathbbm{1}_{\J(W_{k})}\right)
\]
\begin{equation*}
\begin{split}
  &\cong \Hom_{\Mp(W_{k})}\left(\left(\left(\pi_{k+1}\big|_{\J(W_k)}\right)_{Z,\psi}\otimes\overline{\omega}_\psi\right)_{\Hei(W_k)}\otimes\wt\pi_{k}, \mathbbm{1}_{\Mp(W_{k})}\right) \\
  &\cong \Hom_{\Mp(W_{k})}\left(\mathcal{G}_{\psi}\left(\left(\pi_{k+1}\big|_{\J(W_k)}\right)_{Z,\psi}\right)\otimes\left(\omega_\psi\otimes\overline{\omega}_\psi\right)_{\Hei(W_k)}\otimes\wt\pi_{k}, \mathbbm{1}_{\Mp(W_{k})}\right).
\end{split}
\end{equation*}
Here, the subscript ${}_{Z,\psi}$ means taking the maximal quotient such that the center $Z$ of the Heisenberg group $\Hei(W_k)$ acts by $\psi$, and in the last isomorphism, we make use of \cite[Proposition 4.2]{MR2891317}. An easy computation shows that $\left(\omega_\psi\otimes\overline{\omega}_\psi\right)_{\Hei(W_k)}\cong\BC$. Therefore the last Hom space above is isomorphic to
\[
\begin{split}
\Hom_{\Mp(W_{k})}\left(\mathcal{G}_{\psi}\left(\left(\pi_{k+1}\big|_{\J(W_k)}\right)_{Z,\psi}\right)\otimes\wt\pi_{k}, \mathbbm{1}_{\Mp(W_{k})}\right) 
  & \cong \Hom_{\Mp(W_{k})}\left(\mathcal{G}_{\psi}\left(\left(\pi_{k+1}\big|_{\J(W_k)}\right)_{Z,\psi}\right),\wt\pi_{k}^\vee\right)\\
  & \cong \Hom_{\J(W_{k})}\left(\pi_{k+1}\big|_{\J(W_k)}, \mathcal{F}_{\psi}\left(\wt \pi_{k}^\vee\right)\right).
\end{split}
\]
Our claim in the previous paragraph then implies that
\begin{equation}\label{co-rank 2-0}
  \Hom_{\J(W_{k})}\left(\pi_{k+1}, \mathcal{F}_{\psi}\left(\wt\pi_{k}^\vee\right)\right)\neq 0.
\end{equation}
As explicated in \eqref{fj MODEL rewrite II}, this Hom space can be rephrased as a corank two Fourier-Jacobi model.

\vskip 5pt

Next we reduce the corank two model above back to the basic case by appealing to the Mackey theory again. Applying the MVW-involution in \Cref{thm MVW} to $\pi_{k+1}^\vee$ constructed in \eqref{Theta tower O to Sp}, we deduce that
\begin{equation*}
  \pi_{k+1} \cong LQ\left(|\cdot|^{k}\times\cdots\times |\cdot|^{1}\rtimes\pi_1\right).
\end{equation*}
In particular, there is a surjection 
\[
    |\cdot|^k\rtimes\pi_k \twoheadrightarrow \pi_{k+1}.
\]
Substituting above surjection into \eqref{co-rank 2-0}, we deduce that
\begin{equation}\label{FJ corank 2}
    \Hom_{\J(W_{k})}\left(|\cdot|^k\rtimes\pi_k, \mathcal{F}_{\psi}\left(\wt\pi_{k}^\vee\right)\right)\neq 0.
\end{equation}
Let $\mathcal X' = \Sp(W_{k+1})/Q_{Y}$ be the flag variety of $\Sp(W_{k+1})$ consisting of all isotropic lines in $W_{k+1}$, and let $\mathcal E'$ be the $\Sp(W_{k+1})$-equivariant tempered vector bundle whose fiber representation at $Y\in\mathcal{X}'$ is
\[
  \left(|\cdot|^{k} \boxtimes \pi_{k}\right)\otimes \delta_{Q_{Y}}^{\frac{1}{2}}.
\]
Then one has
\[
  \Gamma^\mathcal{S}(\mathcal X', \mathcal E')\cong |\cdot|^{k} \rtimes \pi_{k}
\]
as $\Sp(W_{k+1})$-representations. There are three $\J(W_k)$-orbits on $\mathcal X'$:
\begin{itemize}
  \item the closed orbit $\mathcal Z'$, consisting of the line $Y$, which is a single point;

  \vskip 5pt

  \item the middle orbit $\mathcal M'$, consisting of isotropic lines in $Y^\perp\setminus Y$;
  
  \vskip 5pt
  
  \item the open orbit $\mathcal U'$, consisting of isotropic lines in $W_{k+1}\setminus Y^\perp$.
\end{itemize}
\vskip 5pt
Set $\mathcal V' = \mathcal U'\cup \mathcal M'$. Then there is a $\J(W_k)$-equivariant short exact sequence
\begin{equation*}
  0\longrightarrow
  \Gamma^\mathcal{S}\left(\mathcal V', \mathcal E'\right)
  \longrightarrow
  \Gamma^\mathcal{S}\left(\mathcal X',\mathcal E'\right)
  \longrightarrow
  \Gamma^\mathcal{S}_{\mathcal Z'}\left(\mathcal X',\mathcal E'\right)
  \longrightarrow 0,
\end{equation*}
where $\Gamma^\mathcal{S}_{\mathcal Z'}(\mathcal X',\mathcal E') \cong \Gamma^\mathcal{S}(\mathcal X',\mathcal E')/\Gamma^\mathcal{S}(\mathcal V', \mathcal E')$. This induces an exact sequence of Hom spaces
\begin{equation*}
  \Hom_{\J(W_{k})}\left(\Gamma^\mathcal{S}_{\mathcal Z'}(\mathcal X',\mathcal E'),
    \mathcal{F}_\psi(\wt\pi_k^\vee)\right)
  \longrightarrow
  \Hom_{\J(W_{k})}\left(\Gamma^\mathcal{S}\left(\mathcal X',\mathcal E'\right),\mathcal{F}_\psi(\wt\pi_k^\vee)\right)
  \longrightarrow
  \Hom_{\J(W_{k})}\left(\Gamma^\mathcal{S}(\mathcal V', \mathcal E'), \mathcal{F}_\psi(\wt\pi_k^\vee)\right).
\end{equation*}
By \Cref{P:BorelLemma}, there is a filtration on $\Gamma^\mathcal{S}_{\mathcal Z'}(\mathcal X',\mathcal E')$.
One checks easily that the center $Z(\J(W_k))$ acts trivially on each graded piece
\begin{equation*}
  \Gamma^\mathcal{S}\left(
    \mathcal{Z'}, \Sym^j\mathcal{N}_{\mathcal{Z'}|\mathcal{X'}}^*\otimes\mathcal{E'}\big|_\mathcal{Z'}\right),
\end{equation*}
for all $j\in\mathbb{N}$. On the other hand, the center $Z(\J(W_k))$ acts on $\mathcal{F}_\psi\left(\wt\pi_k^\vee\right)$ by the non-trivial character $\psi$. By comparing central characters, we obtain that
\[
  \Hom_{\J(W_{k})}\left(\Gamma^\mathcal{S}_{\mathcal Z'}(\mathcal X',\mathcal E'),
    \mathcal{F}_\psi(\wt\pi_k^\vee)\right) = 0.
\]
Together with \eqref{FJ corank 2} and above exact sequence of Hom spaces, this implies
\begin{equation}\label{FJ second piece}
  \Hom_{\J(W_{k})}\left(\Gamma^\mathcal{S}(\mathcal V', \mathcal E'), \mathcal{F}_\psi(\wt\pi_k^\vee)\right) \neq 0.
\end{equation}
The decomposition $\mathcal V'=\mathcal M'\cup \mathcal U'$ also yields a $\J(W_k)$-equivariant short exact sequence
\begin{equation*}
  0\longrightarrow
  \Gamma^\mathcal{S}\left(\mathcal U', \mathcal E'\right)
  \longrightarrow
  \Gamma^\mathcal{S}\left(\mathcal V',\mathcal E'\right)
  \longrightarrow
  \Gamma^\mathcal{S}_{\mathcal M'}\left(\mathcal V',\mathcal E'\right)
  \longrightarrow 0,
\end{equation*}
where $\Gamma^\mathcal{S}_{\mathcal M'}(\mathcal V',\mathcal E')\cong \Gamma^\mathcal{S}(\mathcal V',\mathcal E')/\Gamma^\mathcal{S}(\mathcal U', \mathcal E')$, and hence induce another exact sequence of Hom spaces
\begin{equation*}
  \Hom_{\J(W_{k})}\left(
    \Gamma^\mathcal{S}_{\mathcal M'}(\mathcal V',\mathcal E'),
    \mathcal{F}_\psi(\wt\pi_k^\vee)
  \right)
  \longrightarrow
  \Hom_{\J(W_{k})}\left(
    \Gamma^\mathcal{S}(\mathcal V',\mathcal E'),
    \mathcal{F}_\psi(\wt\pi_k^\vee)
 \right)
  \\
  \longrightarrow
  \Hom_{\J(W_{k})}\left(
    \Gamma^\mathcal{S}(\mathcal U', \mathcal E'),
    \mathcal{F}_\psi(\wt\pi_k^\vee)
  \right).
\end{equation*}
Once again we claim that
\begin{equation}\label{contribution open FJIII}
  \Hom_{\J(W_{k})}\left(
    \Gamma^\mathcal{S}\left(\mathcal U', \mathcal E'\right),
    \mathcal{F}_\psi\left(\wt\pi_k^\vee\right)
  \right)\neq 0.
\end{equation}
Note that by \eqref{FJ second piece} and above exact sequence, to show our claim, it suffices to show that
\[
    \Hom_{\J(W_{k})}\left(
    \Gamma^\mathcal{S}_{\mathcal M'}(\mathcal V',\mathcal E'),
    \mathcal{F}_\psi(\wt\pi_k^\vee)
  \right)=0.
\]
We now analyze $\Gamma^\mathcal{S}_{\mathcal M'}(\mathcal V',\mathcal E')$ using Borel's lemma \Cref{P:BorelLemma}. We fix two vectors $e_{k},f_{k}\in W_{k}$, such that they form a (symplectic) hyperbolic plane, and we have a decomposition
\[
  W_{k} = \left\langle e_{k}\right\rangle \oplus W_{k-1} \oplus \left\langle f_{k}\right\rangle.
\]
Choose the representative $Y'=\left\langle e_k\right\rangle\in \mathcal M'$. By \eqref{Parabolic compatible}, we know that $P_{Y'}\coloneqq \Stab_{\J(W_k)}\left(Y'\right)$ is a ``parabolic subgroup'' of $\J(W_k)$
with ``Levi quotient'' $M_{Y'}\cong \GL_1(\BR)\times \J(W_{k-1})$. By computations similar to those in the previous case, one obtains
\[
  \Gamma^\mathcal{S}\left(\mathcal{M'},\mathcal{E'}\big|_\mathcal{M'}\right)
  \cong
  |\cdot|^{k+\frac{1}{2}}\rtimes \left(\pi_{k}\big|_{\J(W_{k-1})}\right).
\]
On the other hand, the complexified conormal bundle $\mathcal N_{\mathcal{M'}|\mathcal{V'}}^*$ is a $\J(W_k)$-equivariant bundle over $\mathcal M'$ whose fiber representation at $Y'$ is $\sgn|\cdot|\boxtimes\mathbbm{1}_{\J(W_{k-1})}$. By \Cref{P:BorelLemma}, there is a $\J(W_k)$-equivariant filtration on $\Gamma^\mathcal{S}_{\mathcal M'}(\mathcal V',\mathcal E')$, with graded pieces
\begin{equation}\label{FJ graded}
  \Gamma^\mathcal{S}\left(
    \mathcal{M'}, \Sym^j\mathcal{N}_{\mathcal{M'}|\mathcal{V'}}^*\otimes\mathcal{E'}\big|_\mathcal{M'}
  \right)
  \cong
  \left(\sgn^j|\cdot|^{k+j+\frac{1}{2}}\right)\rtimes \left(\pi_{k}\big|_{\J(W_{k-1})}\right)
  \quad \mbox{for } j\in\mathbb{N}.
\end{equation}
Using \Cref{vanishing lemma} together with \eqref{Theta tower O to Mp}, we deduce that
\[
  \Hom_{\J(W_{k})}\left(
    \left(\sgn^j|\cdot|^{k+j+\frac{1}{2}}\right)\rtimes \left(\pi_{k}\big|_{\J(W_{k-1})}\right),
    \mathcal{F}_\psi(\wt\pi_k^\vee)
  \right)=0
  \quad\text{for all } j\in\mathbb{N},
\]
and therefore $\Hom_{\J(W_{k})}\left(\Gamma^\mathcal{S}_{\mathcal M'}(\mathcal V',\mathcal E'), \mathcal{F}_\psi(\wt\pi_k^\vee) \right)=0$ as desired.

\vskip 5pt

Finally, choose the representative $\left\langle f_{k+1}\right\rangle \in \mathcal U'$. Then $\Stab_{\J(W_k)}(\langle f_{k+1}\rangle)\cong \Sp(W_k)$ and
\[
  \Gamma^\mathcal{S}\left(\mathcal U', \mathcal E'\right)
  \cong
  \unind_{\Sp(W_{k})}^{\J(W_{k})}\left(\left(
    \big(|\cdot|^{k} \boxtimes \pi_{k}\big)\otimes \delta_{Q_{Y'}}^{\frac{1}{2}}\right)\Big|_{\Sp(W_k)}
  \right)
  \cong \unind_{\Sp(W_{k})}^{\J(W_{k})}  \pi_{k}.
\]
Note that there is a natural non-degenerate $\J(W_k)$-invariant pairing 
\[
    \mathcal{F}_\psi\left(\wt\pi_k^\vee\right) \times \mathcal{F}_{\bar\psi}\left(\wt\pi_k\right) \longrightarrow \BC.
\]
By using this pairing and \Cref{Frobenious}, we have
\begin{equation}\label{contribution open FJII}
\begin{split}
  \Hom_{\J(W_k)}\left(
    \Gamma^\mathcal{S}\left(\mathcal U',\mathcal E'\right),
    \mathcal{F}_\psi\left(\wt\pi_k^\vee\right) 
  \right)
  &\cong
  \Hom_{\J(W_k)}\left(
    \unind_{\Sp(W_{k})}^{\J(W_{k})}\pi_{k},
    \mathcal{F}_\psi\left(\wt\pi_k^\vee\right)\right)\\
  &\hookrightarrow
  \Hom_{\J(W_{k})}\left(
  \mathcal{F}_{\bar\psi}\left(\wt\pi_k\right)\otimes\unind_{\Sp(W_{k})}^{\J(W_{k})}\pi_{k}, 
  \mathbbm{1}_{\J(W_k)}\right)\\
  &\cong 
  \Hom_{\Sp(W_{k})}\left(
  \mathcal{F}_{\bar\psi}\left(\wt\pi_k\right)\big|_{\Sp(W_k)}\otimes\pi_{k}, 
  \mathbbm{1}_{\Sp(W_k)}\right)
\end{split}
\end{equation}
Together with \eqref{contribution open FJIII}, this shows that
\[
  \Hom_{\Sp(W_{k})}\left(
    \wt \pi_{k}\otimes \overline{\omega}_{\psi}, \pi_{k}^\vee
  \right)\neq 0,
\]
which completes the proof.
\end{proof}

\vskip 5pt

Applying \Cref{L:ind2} iteratively, we conclude that  $\Hom_{\Sp(W)}\left(\wt \pi\boxtimes\overline{\omega}_{\psi}, \pi^\vee\right) \neq 0$. This completes the proof of \Cref{P:RtoL}.

\vskip 10pt

\appendix
\bibliographystyle{alpha}
\bibliography{CCZ}
\end{document}